\documentclass[12pt]{amsart}

\usepackage[utf8]{inputenc}
\usepackage[T1,T2A]{fontenc}
\usepackage[english]{babel}

\usepackage{todonotes}

\usepackage{amssymb}
\usepackage{amsmath}
\usepackage{amsthm}
\usepackage{multirow}
\usepackage{longtable}
\usepackage{enumitem}
\usepackage{array}
\usepackage[hidelinks]{hyperref}
\usepackage[all]{xy}
\usepackage{caption}
\usepackage{float}

\newcommand{\CC}{\mathbb{C}}
\newcommand{\ZZ}{\mathbb{Z}}
\newcommand{\QQ}{\mathbb{Q}}

\newcommand{\CCC}{\mathfrak{C}}
\newcommand{\g}{\mathfrak{g}}
\renewcommand{\phi}{\varphi}
\newcommand{\eps}{\varepsilon}

\newcommand{\Sym}{\mathrm{S}}
\newcommand{\Ima}{\operatorname{Im}}
\newcommand{\rk}{\operatorname{rk}}
\newcommand{\Hom}{\operatorname{Hom}}
\newcommand{\tr}{\operatorname{tr}}
\newcommand{\cone}{\operatorname{cone}}

\newcommand{\GL}{\operatorname{GL}}
\newcommand{\Sp}{\operatorname{Sp}}
\newcommand{\SL}{\operatorname{SL}}
\newcommand{\SO}{\operatorname{SO}}
\newcommand{\Spin}{\operatorname{Spin}}
\newcommand{\Gtwo}{\operatorname{G}_2}
\newcommand{\E}{\operatorname{E}}

\newcommand{\gl}{\mathfrak{gl}}

\newcommand{\pairing}[2]{\langle#1,\,#2\rangle}
\newcommand{\dotprod}[2]{(#1,\,#2)}
\newcommand{\symplprod}[2]{\{#1,\,#2\}}
\newcommand{\floortwo}[1]{{\lfloor{#1}\rfloor}_2}

\newcommand{\transpose}[1]{{#1}^{\intercal}}

\newtheorem{theorem}{Theorem}[section]
\newtheorem*{theorem*}{Theorem}
\newtheorem{lemma}[theorem]{Lemma}
\newtheorem{corollary}[theorem]{Corollary}

\theoremstyle{remark}
\newtheorem{remark}[theorem]{Remark}

\theoremstyle{definition}
\newtheorem{example}{Example}[section]
\newtheorem*{example*}{Example}

\newcounter{num}[table]
\newcommand{\newcase}{\refstepcounter{num}\arabic{num}}

\renewcommand{\tabcolsep}{3pt}

\begin{document}

\date{}
\title[Orbits of spherical representations and Pyasetskii duality]{Orbits of spherical representations and Pyasetskii duality}

\author{Daniil Shunin}
\address{
Lomonosov Moscow State University, Moscow, Russia
\newline\indent HSE University, Moscow, Russia}
\email{sshunindaniil@gmail.com}
\thanks{This article is an output of a research project implemented as part of the Basic Research Program at HSE University}

\keywords{spherical representation, spherical module, Pyasetskii duality, orbit, conormal bundle}

\begin{abstract}
Dual representations $V$ and $V^*$ of a complex connected algebraic group $G$ simultaneously have either infinitely or finitely many orbits. Whenever the latter holds, the orbits in $V$ and $V^*$ are in a bijective correspondence called \textit{Pyasetskii duality}~\cite{Py}. We obtain a complete description of this duality in the case of spherical representations.
\end{abstract}

\maketitle


\section{Introduction}

Let $G$ be a connected linear algebraic group defined over the field $\CC$. It was shown in~\cite{Py} that a finite-dimensional representation $V$ of $G$ and its dual $V^*$ both contain either finitely or infinitely many orbits. In the first case there is a bijection between the orbits in mutually dual representations called \textit{Pyasetskii duality}.

From now on, $G$ is assumed to have finitely many orbits in $V$. The reduced zero fiber of the moment map $\mu: V\oplus V^* \to\g^*$ is called the \textit{commuting variety}, that is,
\begin{equation*}
\CCC = \mu^{-1}(0)_{red} = \{(v, v^*)\mid v^*\in(\g v)^\perp\} = \{(v, v^*)\mid v\in(\g v^*)^\perp\}.
\end{equation*}
It is an unmixed $G$-stable closed subvariety of dimension $\dim V$. Note that $\CCC\cap(O\times V^*)$ is the total space $N^*O$ of the conormal bundle of an orbit $O\subset V$. Then $\CCC$ is the union $\cup N^*O$ over all orbits $O$ in $V$ (by symmetry, in $V^*$). The following theorem explains the role of the irreducible components of the commuting variety:

\begin{theorem}[{\cite{Py}}]
Suppose $G$ has finitely many orbits in $V$ and $\CCC=\cup_i \CCC_i$ is the irreducible decomposition. Then

1. $\dim \CCC_i = \dim V$ for each i;

2. the projection of $\CCC$ onto the first (resp. second) factor sets up a one-to-one correspondence between the irreducible components of $\CCC$ and the $G$-orbits in $V$ (resp. $V^*$). Namely, each $\CCC_i$ arises as the closure of $N^*O$ for a unique orbit $O$ in $V$ (by symmetry, $Q$ in $V^*$).
\end{theorem}

The orbits $O\subset V$ and $Q\subset V^*$ defined by the condition $\overline{N^*O}=\overline{N^*Q}=\CCC_i$ are called (\textit{Pyasetskii-}) \textit{dual} and denoted by $O^\vee=Q$. Observe that both $\{0\}\times V^*$ and $V\times\{0\}$ occur as irreducible components of $\CCC$, so the zero orbit is always dual to the open one.

While this orbit correspondence may be quite complicated in general, a complete answer for a special case was obtained in~\cite{Pan}.

\begin{theorem}[\cite{Pan}]
\label{thm: Pan}
Let $\mathfrak{l}$ be a simple $\ZZ$-graded Lie algebra with only three nonzero parts $\mathfrak{l}=\mathfrak{l}_{-1}\oplus\mathfrak{l}_0\oplus\mathfrak{l}_1$. Denote by $G=L_0$ a connected group corresponding to $\mathfrak{l}_0$ and $V=\mathfrak{l}_1$. Then

1. the representation $G : V$ is irreducible;

2. there are finitely many orbits $O_0,\dots,O_r$ in $V$, and $O_i\subset\overline{O}_j$ if and only if $i\leq j$;

3. there are finitely many orbits $Q_0,\dots,Q_r$ in $V^*$, and $Q_i\subset\overline{Q}_j$ if and only if $i\leq j$;

4. $(O_i)^\vee=Q_{r-i}$ in the Pyasetskii duality.
\end{theorem}

The actions arising in this way are called \textit{Abelian}, they form a special case of a more general class of spherical representations. A finite-dimensional representation $V$ of a connected reductive group $G$ is called \textit{spherical} (\textit{multiplicity-free}) if the multiplicity of any simple $G$-module in the coordinate algebra $\CC[V]$ is at most one. As finiteness of the number of orbits is one of numerous significant properties of such representations, they possess Pyasetskii duality.

A complete classification of all spherical modules was obtained in~\cite{Kac,BeRa,Lea}. An important step of this classification is to reduce the class of representations in question. The terminology introduced below follows Knop, see~\cite[\S\,5]{Kn_MFS} (groups are \textit{not} assumed to be reductive).

Consider two representations $\rho_1: G_1\to\GL(V_1)$ and $\rho_2: G_2\to\GL(V_2)$. We say that the pairs $(G_1,V_1)$ and $(G_2,V_2)$ are \textit{geometrically equivalent} if there is an isomorphism $\phi: V_1\to V_2$ identifying groups $\rho_1(G_1)$ and $\rho_2(G_2)$. Notice that this implies geometrical equivalence of $(G_1,V_1^*)$ and $(G_2,V_2^*)$. Next, a representation $(G,V)$ is \textit{decomposable} if it is geometrically equivalent to a representation of the form $(G_1\times G_2,V_1\oplus V_2)$ with non-zero representations $V_1$ and $V_2$ of the corresponding groups $G_1$ and $G_2$. Otherwise, it is called \textit{indecomposable}. Finally, a representation $\rho: G\to\GL(V)$ is called \textit{saturated} if the dimension of the center of $\rho(G)$ equals the number of irreducible summands of $V$.
\begin{remark}
\label{rm: equivalent_orbits}
In the case where $G$ is reductive, representations $(G,V)$ and $(G,V^*)$ are geometrically equivalent. Identifying the spaces of mutually dual representations one may consider the Pyasetskii duality as a pairing between $G$-orbits in $V$. We retain the notation $O^\vee\subset V$ for the orbit dual to $O\subset V$.
\end{remark}
The study of an arbitrary spherical representation may be reduced to that of indecomposable and saturated ones (listed in~\cite[\S 5]{Kn_MFS}). Further we show that the problem of describing Pyasetskii duality complies with this reduction.

Our aim is to determine pairs of dual orbits in any spherical representation. Clearly, it suffices to handle only one representation in each class of geometrical equivalence. Next, every spherical representation of a connected reductive group $G'$ is obtained by restriction from a sum of indecomposable saturated representations of a bigger group $G$ normalizing $G'$ (cf.~\cite{Lea}). The following lemma shows that the construction of Pyasetskii for $(G,V)$ and $(G',V)$ leads to the same duality correspondence.

\begin{lemma}
\label{lemma: same_orbits}
Let $V$ be a representation of a connected group $G$, and $G'$ be its normal subgroup with finitely many orbits in $V$. Then $G$ and $G'$ have the same orbits.
\end{lemma}

We provide the next simple lemma which readily implies that the problem is reduced to describing the orbit duality only for indecomposable saturated representations.
\begin{lemma}
\label{lemma: duality_and_product}
Let $(G_1,V_1)$ and $(G_2,V_2)$ be two representations of connected groups with finite number of orbits. Then $(O_1\times O_2)^\vee = O_1^\vee\times O_2^\vee$ for any $G_i$-obits $O_i\in V_i$, $i=1,2$.
\end{lemma}
In what follows, we use ``orbit diagram'' as a synonym for ``Hasse diagram for the poset of orbit closures''. Below is the main result of the paper.

\begin{theorem}
\label{thm: main_theorem}
Orbit diagram and dual orbits for every (up to geometric equivalence) indecomposable saturated \textup{(a)} irreducible and \textup{(b)} reducible spherical representation are listed in \textup{Tables}~\textup{\ref{table: irreducible_modules}} and~\textup{\ref{table: reducible_modules}}, respectively.
\end{theorem}

We notice that the order reversing property of Pyasetskii duality, which takes place in the case of abelian actions, does not hold in general.

\bigbreak
\noindent
\label{expl: tables}
\textit{Explanation for the tables:}
Each group $\GL_n$, $\SO_n$, $\Sp_n$ acts on $\CC^n$ naturally; $\CC^\times$ acts on $V$ by scaling.

In each case of Table~\ref{table: irreducible_modules} we assume that every non-abelian factor of $G$ acts on the corresponding tensor factor of $V$. In rows~\ref{tablecase: irr_E6} and~\ref{tablecase: irr_G2} the first factor acts on $V$ via faithful representation of minimal dimension. In rows~\ref{tablecase: irr_spin7} and~\ref{tablecase: irr_spin9} the first factor acts on $V$ via the spin representation, whereas in row~\ref{tablecase: irr_spin10} it acts via a half-spin representation.

In Table~\ref{table: reducible_modules}, the pair $(G,V)$ is arranged in two levels, with $G$ on the upper level and $V$ on the lower one.  Each factor of $G$ acts diagonally on all components of V with which it is connected by an
edge. In the last row the symbol $\CC_\pm^8$ stands for the spaces of the two half-spin representations
of $\Spin_8$.

The last column of both tables contains the orbit diagrams along with the indication of dual orbits (by definition, $O\leq O'$ if $O\subset \overline{O'}$). We identify orbits in $V$ and in $V^*$ according to Remark~\ref{rm: equivalent_orbits}. Specifically, by replacing $G$ with its finite cover one may assume that, with $B\subset G$ and $T\subset B$ denoting a fixed Borel subgroup and a maximal torus, there is an involution $\theta$ on $G$ (\textit{Weyl involution}) such that $\theta(B)\cap B=T$ and $\theta(t)=t^{-1}$ for any $t\in T$. One can show that $V$ equipped with a twisted $G$-action $g\cdot v=\theta(g)v$ is isomorphic to the dual $G$-module $V^*$. The identification of orbits in $V$ and $V^*$ by this isomorphism is independent of the choice of $B$ and $T$.

In Table~\ref{table: reducible_modules}, $V$ is decomposed into a direct sum of two simple submodules $V_1$ and $V_2$, and the letters $O$ and $O'$ (with subscripts) denote orbits in $V_1$ and $V_2$, respectively. The notation for the remaining orbits is as follows: given $O\subset V_1$ and $O'\subset V_2$, the product $O\times O'$ either is a single $G$-orbit $Y$ or decomposes into a union $Y\cup Z$ of two disjoint $G$-invariant subvarieties, with $Y$ being an open orbit in $O\times O'$. Except for the case~\ref{tablecase: red_SpxCxC}, $Z$ is also a single orbit. In the case~\ref{tablecase: red_SpxCxC}, the set $Z$ splits into two orbits: an open orbit $Z^\circ$ and an orbit $Z_\sim$. The signs of orbits are equipped with subscripts. The subscripts at the signs $Y$ and $Z$ are coherent with the subscripts at the corresponding $O$ and $O'$.

\begin{remark}
\label{rm: orbit_graph}
In the first and second row of Table~\ref{table: reducible_modules}, the orbit diagram depends on the parity of $n$. Namely, the expression in parenthesis should be taken into account if $n$ is odd. Otherwise, $Z_k$ is defined to be empty, and the correct inclusions are: in the first row $O_k,Y_{k-1}\subset\overline{Y_k}$; in the second row $O_k\subset\overline{Y_k}$.
\end{remark}

The paper is organized as follows.
In \S\,\ref{s: properties_of_dual_orbits} we provide simple properties of dual orbits and prove Lemmas~\ref{lemma: same_orbits} and~\ref{lemma: duality_and_product}.
In \S\,\ref{s: orbits} we recall some notions of spherical theory and then obtain explicit orbit decompositions for representations which are geometrically equivalent to those in the tables.
In \S\,\ref{s: proof_for_irreducible} and \S\,\ref{s: proof_for_reducible} parts \textup{(a)} and \textup{(b)} of Theorem~\ref{thm: main_theorem} are proved.

The author is grateful to his scientific advisor D.\,A.~Timashev for the formulation of the problem, his constant attention, and numerous valuable comments and suggestions.

\subsection*{Notation and conventions}

The base field is that of complex numbers $\CC$. All varieties, groups, and subgroups are assumed to be algebraic; the corresponding Lie algebras are denoted by Gothic letters. All topological terms refer to the Zariski topology. All linear representations are assumed to be rational and finite-dimensional.

Below $X$ denotes a smooth subvariety of a vector space $V$.


$\Sym^2 V$ is the symmetric square of a vector space $V$;

$\wedge^2 V$ is the exterior square of a vector space $V$;

$U^\perp$ is the annihilator of $U$ in $V^*$, or, if $V$ is equipped with a (skew-)symmetric bilinear form, the space of the vectors orthogonal to $U$;

$x\perp y$ means that $x\in\{y\}^\perp$ in the preceding notation;

$\overline{X}$ is the closure of $X$ in $V$;

$T_x X\subset V$ is the tangent space of $X$ at a point $x$;

$N^*_x X = (T_x X)^\perp\subset V^*$ is the conormal space of $X$ at a point $x$;

$N^* X = \{(x,y)\mid x\in X, y\in N^*_x X\}\subset V\oplus V^*$ is the conormal bundle of $X$;

$V/G$ is the set of orbits in the space $V$ acted on by a group $G$;

$\cone(u_1,\dots,u_k)$ is the polyhedral cone in a real vector space generated by vectors $u_i$;

$\transpose{x}$ is the transpose matrix, or the adjoint operator $W^*\to U^*$ for $x: U\to W$;

{
\footnotesize

\begin{table}[H]
\caption{Irreducible representations}
\label{table: irreducible_modules}

\begin{tabular}{|c|c|c|c|c|}

\hline

No. & $G$ & $V$ & Note & Orbits and duality \\

\hline

\newcase\label{tablecase: irr_GL_GL}
&
$\GL_q \times \GL_p$
&
$\CC^q \otimes \CC^p$
& 
\renewcommand{\tabcolsep}{0pt}%
\begin{tabular}{c}
     $q\geq p\geq 1,$ \\
     $r = p$
\end{tabular}
\renewcommand{\tabcolsep}{3pt}%
&
\multirow{5}{*}{
\begin{tabular}{c}
$O_i^\vee = O_{r-i}$ \\
\unitlength 0.70ex
\linethickness{0.4pt}
\begin{picture}(26.00,4.00)
\put(2.00,1.5){\makebox(0,0)[cc]{$O_0$}}
\put(13.00,1.5){\makebox(0,0)[cc]{${\cdots}$}}
\put(24.00,1.5){\makebox(0,0)[cc]{$O_r$}}
\put(5.00,1.5){\vector(1,0){5.00}}
\put(16.00,1.5){\vector(1,0){5.00}}
\end{picture} 
\end{tabular}
}
\\

\cline{1-4}

\newcase\label{tablecase: irr_GL_Sym}
& $\GL_n$ & $\Sym^2\CC^n$ & $n\geq 1,r=n$ &
\\

\cline{1-4}

\newcase\label{tablecase: irr_GL_Ext}
& $\GL_n$ & $\mathrm{\wedge}^2\CC^n$ & 
\renewcommand{\tabcolsep}{0pt}%
\begin{tabular}{c}
     $n\geq 2,$ \\
     $r = \lfloor n/2\rfloor$
\end{tabular}
\renewcommand{\tabcolsep}{3pt}%
&
\\

\cline{1-4}

\newcase\label{tablecase: irr_E6}
& $\E_6\times\CC^\times$ & $\CC^{27}$ & $r=3$ &
\\

\hline

\newcase\label{tablecase: irr_SO}
& $\SO_n\times\CC^\times$ & $\CC^n$ & $n\geq 2$ &
\multirow{4}{*}{
\begin{tabular}{c}
$O_0^\vee=O_2, O_1^\vee = O_1$ \\
\unitlength 0.70ex
\linethickness{0.4pt}
\begin{picture}(26.00,4.00)
\put(2.00,1.5){\makebox(0,0)[cc]{$O_0$}}
\put(13.00,1.5){\makebox(0,0)[cc]{$O_1$}}
\put(24.00,1.5){\makebox(0,0)[cc]{$O_2$}}
\put(5.00,1.5){\vector(1,0){5.00}}
\put(16.00,1.5){\vector(1,0){5.00}}
\end{picture}
\end{tabular}
}
\\

\cline{1-4}

\newcase\label{tablecase: irr_G2}
& $\Gtwo\times\CC^\times$ & $\CC^7$ & $-$ &
\\

\cline{1-4}

\newcase\label{tablecase: irr_spin7}
& $\Spin_7\times\CC^\times$ & $\CC^{8}$ & $-$ &
\\

\cline{1-4}

\newcase\label{tablecase: irr_spin10}
& $\Spin_{10}\times\CC^\times$ & $\CC^{16}$ & $-$ &
\\

\hline

\newcase\label{tablecase: irr_spin9}
& $\Spin_9\times\CC^\times$ & $\CC^{16}$ & $-$ &
\begin{tabular}{c}
$O_0^\vee=O_3, O_{1}^\vee=O_{1}, O_{2}^\vee=O_{2}$ \\
\unitlength 0.70ex
\linethickness{0.4pt}
\unitlength 0.70ex
\linethickness{0.4pt}
\begin{picture}(37.00,4.00)(-18.50,-1.00)
\put(-16.50,0.00){\makebox(0,0)[cc]{$O_{0}$}}
\put(-5.50,0.00){\makebox(0,0)[cc]{$O_{1}$}}
\put(5.50,0.00){\makebox(0,0)[cc]{$O_{2}$}}
\put(16.50,0.00){\makebox(0,0)[cc]{$O_{3}$}}
\put(-13.50,0.00){\vector(1,0){5.00}}
\put(-2.50,0.00){\vector(1,0){5.00}}
\put(8.50,0.00){\vector(1,0){5.00}}
\end{picture}
\end{tabular}
\\

\hline

\newcase\label{tablecase: irr_Sp_first}
& $\Sp_{2n} \times\CC^\times$ & $\CC^{2n}$ & $n\geq 1$ &
\unitlength 0.70ex
\linethickness{0.4pt}
\begin{picture}(15.00,3.00)(-7.00,-1.00)
\put(-5.50,0.00){\makebox(0,0)[cc]{$O_0$}}
\put(5.50,0.00){\makebox(0,0)[cc]{$O_1$}}
\put(-2.50,0.00){\vector(1,0){5.00}}
\end{picture}
\\

\hline

\newcase\label{tablecase: irr_Sp_GL2}
& $\Sp_{2n}\times\GL_2$ & $\CC^{2n}\otimes\CC^2$ & $n\geq 2$ &
\begin{tabular}{c}
$O_{00}^\vee=O_{22}, O_{10}^\vee=O_{10}, O_{20}^\vee=O_{20}$ \\
\unitlength 0.70ex
\linethickness{0.4pt}
\begin{picture}(37.00,4.00)(-18.50,-1.00)
\put(-16.50,0.00){\makebox(0,0)[cc]{$O_{00}$}}
\put(-5.50,0.00){\makebox(0,0)[cc]{$O_{10}$}}
\put(5.50,0.00){\makebox(0,0)[cc]{$O_{20}$}}
\put(16.50,0.00){\makebox(0,0)[cc]{$O_{22}$}}
\put(-13.50,0.00){\vector(1,0){5.00}}
\put(-2.50,0.00){\vector(1,0){5.00}}
\put(8.50,0.00){\vector(1,0){5.00}}
\end{picture}
\end{tabular}
\\

\hline

\newcase\label{tablecase: irr_Sp_GL3}
& $\Sp_{2n}\times\GL_3$ & $\CC^{2n}\otimes\CC^3$ & $n\geq 3$ &
\begin{tabular}{c}
$O_{00}^\vee=O_{32}, O_{10}^\vee=O_{22}, O_{20}^\vee=O_{30}$ \\
\unitlength 0.70ex
\linethickness{0.4pt}
\begin{picture}(48.00,13.00)(-24.00,-6.00)
\put(-22.00,0.00){\makebox(0,0)[cc]{$O_{00}$}}
\put(-11.00,0.00){\makebox(0,0)[cc]{$O_{10}$}}
\put(0.00,0.00){\makebox(0,0)[cc]{$O_{20}$}}
\put(11.00,5.00){\makebox(0,0)[cc]{$O_{22}$}}
\put(11.00,-5.00){\makebox(0,0)[cc]{$O_{30}$}}
\put(22.00,0.00){\makebox(0,0)[cc]{$O_{32}$}}
\put(-19.00,0.00){\vector(1,0){5.00}}
\put(-8.00,0.00){\vector(1,0){5.00}}
\put(3.00,1.00){\vector(3,2){5.00}}
\put(3.00,-1.00){\vector(3,-2){5.00}}
\put(14.00,4.00){\vector(3,-2){5.00}}
\put(14.00,-4.00){\vector(3,2){5.00}}
\end{picture}   
\end{tabular}
\\

\hline

\newcase & $\Sp_4\times\GL_3$ & $\CC^4\otimes\CC^3$ & $-$ & 
\begin{tabular}{c}
$O_{00}^\vee=O_{32}, O_{10}^\vee=O_{22}, O_{20}^\vee=O_{20}$ \\
\unitlength 0.70ex
\linethickness{0.4pt}
\begin{picture}(48.00,4.00)(-24.00,-1.00)
\put(-22.00,0.00){\makebox(0,0)[cc]{$O_{00}$}}
\put(-11.00,0.00){\makebox(0,0)[cc]{$O_{10}$}}
\put(0.00,0.00){\makebox(0,0)[cc]{$O_{20}$}}
\put(11.00,0.00){\makebox(0,0)[cc]{$O_{22}$}}
\put(22.00,0.00){\makebox(0,0)[cc]{$O_{32}$}}
\put(-19.00,0.00){\vector(1,0){5.00}}
\put(-8.00,0.00){\vector(1,0){5.00}}
\put(3.00,0.00){\vector(1,0){5.00}}
\put(14.00,0.00){\vector(1,0){5.00}}

\end{picture}   
\end{tabular}
\\

\hline

\newcase\label{tablecase: irr_Sp_last}
& $\Sp_4\times\GL_n$ & $\CC^4\otimes\CC^n$ & $n\geq 4$ & 
\begin{tabular}{c}
$O_{00}^\vee=O_{44}$,
$O_{10}^\vee = O_{32},
O_{20}^\vee = O_{20},
O_{22}^\vee = O_{22}$ \\
\unitlength 0.70ex
\linethickness{0.4pt}
\begin{picture}(60.00,4.00)(-29.50,-1.00)
\put(-27.50,0.00){\makebox(0,0)[cc]{$O_{00}$}}
\put(-16.50,0.00){\makebox(0,0)[cc]{$O_{10}$}}
\put(-5.50,0.00){\makebox(0,0)[cc]{$O_{20}$}}
\put(5.50,0.00){\makebox(0,0)[cc]{$O_{22}$}}
\put(16.50,0.00){\makebox(0,0)[cc]{$O_{32}$}}
\put(27.50,0.00){\makebox(0,0)[cc]{$O_{44}$}}
\put(-24.50,0.00){\vector(1,0){5.00}}
\put(-13.50,0.00){\vector(1,0){5.00}}
\put(-2.50,0.00){\vector(1,0){5.00}}
\put(8.50,0.00){\vector(1,0){5.00}}
\put(19.50,0.00){\vector(1,0){5.00}}
\end{picture}
\end{tabular}
\\
\hline

\end{tabular}
\end{table}
}


{
\footnotesize

\begin{longtable}{|c|c|c|c|}
\caption{Reducible representations}
\label{table: reducible_modules}
\endfirsthead
\caption{Reducible representations (continued)}
\endhead

\hline

No. & $(G,V)$ & Note & Orbits and duality \\

\hline

\newcase\label{tablecase: red_GLxC} &
\begin{tabular}{c}
\unitlength 0.70ex
\linethickness{0.4pt}
\begin{picture}(16.00,12.00)(-9.00,-9.00)
\put(0.00,0.00){\makebox(0,0)[cc]{$\GL_n\times\CC^\times$}}
\put(-1.00,-8.00){\makebox(0,0)[cc]{$\wedge^2 \CC^n\oplus \CC^n$}}
\put(-5.00,-2.00){\line(0,-1){4.00}}
\put(-3.00,-2.00){\line(3,-2){6.00}}
\put(4.00,-2.00){\line(0,-1){4.00}}
\end{picture}
\end{tabular}
&
\renewcommand{\tabcolsep}{0pt}
\begin{tabular}{c}
$n\geq 4$,  \\
cf. \\
rem.~\ref{rm: orbit_graph} 
\end{tabular}
\renewcommand{\tabcolsep}{3pt}
& 
\begin{tabular}{c}
$\{0\}^\vee=Y_k, (O'_1)^\vee=O_{k}, O_i^\vee=Y_{k-i}$, \\
$Z_i^\vee=
\begin{cases}
    Z_{k-i}, & \text{if } n=2k,\\
    Z_{k-i+1}, & \text{if } n=2k+1.
\end{cases}
$\\
\unitlength 0.75ex
\linethickness{0.4pt}
\begin{picture}(75.00,16.00)(-38.00,-6.00)
\put(-35.50,0.00){\makebox(0,0)[cc]{$O'_1$}}
\put(-26.50,0.00){\makebox(0,0)[cc]{$Z_1$}}
\put(-17.50,-5.00){\makebox(0,0)[cc]{$Y_1$}}
\put(-9.00,0.00){\makebox(0,0)[cc]{$Z_2$}}
\put(0.00,-4.00){\makebox(0,0)[cc]{$\cdots$}}
\put(8.50,0.00){\makebox(0,0)[cc]{$Z_{k-1}$}}
\put(17.50,-5.00){\makebox(0,0)[cc]{$Y_{k-1}$}}
\put(26.50,0.00){\makebox(0,0)[cc]{$(Z_k$}}
\put(34.50,0.00){\makebox(0,0)[cc]{$)Y_k$}}

\put(-35.50,8.00){\makebox(0,0)[cc]{$\{0\}\;$}}
\put(-26.50,8.00){\makebox(0,0)[cc]{$O_1$}}
\put(-9.00,8.00){\makebox(0,0)[cc]{$O_2$}}
\put(0.00,8.00){\makebox(0,0)[cc]{$\cdots$}}
\put(8.50,8.00){\makebox(0,0)[cc]{$O_{k-1}$}}
\put(26.50,8.00){\makebox(0,0)[cc]{$O_{k}$}}

\put(-33.50,0.00){\vector(1,0){5.00}}
\put(-24.50,-1.00){\vector(3,-2){4.50}}
\put(-15.50,-4.00){\vector(3,2){4.50}}
\put(-7.00,-1.00){\vector(3,-2){3.50}}
\put(2.00,-3.50){\vector(3,2){3.50}}
\put(10.50,-2.00){\vector(3,-2){3.50}}
\put(19.50,-4.00){\vector(3,2){4.00}}
\put(29.50,0.00){\vector(1,0){3.00}}

\put(-33.50,8.00){\vector(1,0){5.00}}
\put(-24.50,8.00){\vector(1,0){13.00}}
\put(-6.00,8.00){\vector(1,0){2.50}}
\put(2.00,8.00){\vector(1,0){2.50}}
\put(12.00,8.00){\vector(1,0){11.50}}

\put(-36.00,6.00){\vector(0,-1){4.00}}
\put(-27.00,6.00){\vector(0,-1){4.00}}
\put(-9.5,6.00){\vector(0,-1){4.00}}
\put(7.00,6.00){\vector(0,-1){4.00}}
\put(26.00,6.00){\vector(0,-1){4.00}}

\end{picture}   
\end{tabular}
\\

\hline

\newcase\label{tablecase: red_GLxC*} &
\begin{tabular}{c}
\unitlength 0.70ex
\linethickness{0.4pt}
\begin{picture}(16.00,12.00)(-9.00,-9.00)
\put(0.00,0.00){\makebox(0,0)[cc]{
$\GL_n\times\CC^\times$}
}
\put(-1.00,-8.00){\makebox(0,0)[cc]{
$\wedge^2 \CC^{n}\oplus (\CC^{n})^*$
}}
\put(-5.00,-2.00){\line(0,-1){4.00}}
\put(-3.00,-2.00){\line(3,-2){6.00}}
\put(4.00,-2.00){\line(0,-1){4.00}}
\end{picture}
\end{tabular}
&
\renewcommand{\tabcolsep}{0pt}
\begin{tabular}{c}
$n\geq 4$,  \\
cf. \\
rem.~\ref{rm: orbit_graph} 
\end{tabular}
\renewcommand{\tabcolsep}{3pt}
& 
\begin{tabular}{c}
$\{0\}^\vee=Y_k, (O'_1)^\vee=O_{k}, O_i^\vee=Y_{k-i}$, \\
$Z_i^\vee=
\begin{cases}
    Z_{k-i}, & \text{if } n=2k,\\
    Z_{k-i+1}, & \text{if } n=2k+1.
\end{cases}
$\\
\unitlength 0.70ex
\linethickness{0.4pt}
\begin{picture}(54.00,22.00)(-27.00,-19.00)
\put(-25.00,0.00){\makebox(0,0)[cc]{$\{0\}$}}
\put(-12.50,0.00){\makebox(0,0)[cc]{$O_{1}$}}
\put(0.00,0.00){\makebox(0,0)[cc]{$\cdots$}}
\put(12.50,0.00){\makebox(0,0)[cc]{$O_{k-1}$}}
\put(25.00,0.00){\makebox(0,0)[cc]{$O_{k}$}}

\put(-25.00,-9.00){\makebox(0,0)[cc]{$O'_1$}}
\put(-12.50,-9.00){\makebox(0,0)[cc]{$Z_{1}$}}
\put(0.00,-9.00){\makebox(0,0)[cc]{$\cdots$}}
\put(12.50,-9.00){\makebox(0,0)[cc]{$Z_{k-1}($}}
\put(25.00,-9.00){\makebox(0,0)[cc]{$\,Z_{k})$}}

\put(-12.50,-18.00){\makebox(0,0)[cc]{$Y_{1}$}}
\put(0.00,-18.00){\makebox(0,0)[cc]{$\cdots$}}
\put(12.50,-18.00){\makebox(0,0)[cc]{$Y_{k-1}$}}
\put(25.00,-18.00){\makebox(0,0)[cc]{$Y_{k}$}}

\put(-20.50,0.00){\vector(1,0){4.00}}
\put(-8.00,0.00){\vector(1,0){4.00}}
\put(4.50,0.00){\vector(1,0){4.00}}
\put(17.00,0.00){\vector(1,0){4.00}}

\put(-20.50,-9.00){\vector(1,0){4.00}}
\put(-8.00,-9.00){\vector(1,0){4.00}}
\put(4.50,-9.00){\vector(1,0){4.00}}
\put(17.00,-9.00){\vector(1,0){4.00}}

\put(-8.00,-18.00){\vector(1,0){4.00}}
\put(4.50,-18.00){\vector(1,0){4.00}}
\put(17.00,-18.00){\vector(1,0){4.00}}

\put(-25.00,-2.50){\vector(0,-1){4.00}}
\put(-12.50,-2.50){\vector(0,-1){4.00}}
\put(-00.00,-2.50){\vector(0,-1){4.00}}
\put(12.50,-2.50){\vector(0,-1){4.00}}
\put(25.00,-2.50){\vector(0,-1){4.00}}

\put(-12.50,-11.50){\vector(0,-1){4.00}}
\put(-00.00,-11.50){\vector(0,-1){4.00}}
\put(12.50,-11.50){\vector(0,-1){4.00}}
\put(25.00,-11.50){\vector(0,-1){4.00}}

\end{picture}   
\end{tabular}
\\

\hline

\multirow{6}{*}{
\newcase\label{tablecase: red_GLxGL}
}
&
\multirow{6}{*}{

\begin{tabular}{c}
\unitlength 0.70ex
\linethickness{0.4pt}
\begin{picture}(24.00,12.00)(-12.00,-9.00)
\put(0.00,0.00){
\makebox(0,0)[cc]{$\GL_q\times\GL_p$}
}
\put(0.00,-8.00){\makebox(0,0)[cc]{$(\CC^{q}\otimes\CC^p)\oplus\CC^{p}$}}
\put(-5.50,-2.00){\line(-2,-3){2.50}}
\put(3.50,-2.00){\line(-2,-3){2.50}}
\put(5.50,-2.00){\line(2,-3){2.50}}
\end{picture}
\end{tabular}

}
&

$q \geq p \geq 2$

&
\begin{tabular}{c}

$\{0\}^\vee=Y_p, (O'_1)^\vee=O_{p}, O_i^\vee=Y_{p-i}$,
$Z_i^\vee=Z_{p-i}.$
\\
\unitlength 0.75ex
\linethickness{0.4pt}
\begin{picture}(75.00,16.00)(-42.00,-6.00)
\put(-35.50,0.00){\makebox(0,0)[cc]{$O'_1$}}
\put(-26.50,0.00){\makebox(0,0)[cc]{$Z_1$}}
\put(-17.50,-5.00){\makebox(0,0)[cc]{$Y_1$}}
\put(-9.00,0.00){\makebox(0,0)[cc]{$Z_2$}}
\put(0.00,-4.00){\makebox(0,0)[cc]{$\cdots$}}
\put(8.50,0.00){\makebox(0,0)[cc]{$Z_{p-1}$}}
\put(17.50,-5.00){\makebox(0,0)[cc]{$Y_{p-1}$}}
\put(26.50,0.00){\makebox(0,0)[cc]{$Y_p$}}

\put(-35.50,8.00){\makebox(0,0)[cc]{$\{0\}\;$}}
\put(-26.50,8.00){\makebox(0,0)[cc]{$O_1$}}
\put(-9.00,8.00){\makebox(0,0)[cc]{$O_2$}}
\put(0.00,8.00){\makebox(0,0)[cc]{$\cdots$}}
\put(8.50,8.00){\makebox(0,0)[cc]{$O_{p-1}$}}
\put(26.50,8.00){\makebox(0,0)[cc]{$O_{p}$}}

\put(-33.50,0.00){\vector(1,0){5.00}}
\put(-24.50,-1.00){\vector(3,-2){4.50}}
\put(-15.50,-4.00){\vector(3,2){4.50}}
\put(-7.00,-1.00){\vector(3,-2){3.50}}
\put(2.00,-3.50){\vector(3,2){3.50}}
\put(10.50,-2.00){\vector(3,-2){3.50}}
\put(19.50,-4.00){\vector(3,2){4.00}}

\put(-33.50,8.00){\vector(1,0){5.00}}
\put(-24.50,8.00){\vector(1,0){13.00}}
\put(-6.00,8.00){\vector(1,0){2.50}}
\put(2.00,8.00){\vector(1,0){2.50}}
\put(12.00,8.00){\vector(1,0){11.50}}

\put(-36.00,6.00){\vector(0,-1){4.00}}
\put(-27.00,6.00){\vector(0,-1){4.00}}
\put(-9.5,6.00){\vector(0,-1){4.00}}
\put(7.00,6.00){\vector(0,-1){4.00}}
\put(26.00,6.00){\vector(0,-1){4.00}}

\end{picture}   

\end{tabular}
\\

\cline{3-4}

&

&

$p > q \geq 2$

&

\begin{tabular}{c}
$\{0\}^\vee=Y_q, (O'_1)^\vee=O_{q}, O_i^\vee=Y_{q-i}$,
$Z_i^\vee=Z_{q-i+1}$.
\\
\unitlength 0.75ex
\linethickness{0.4pt}
\begin{picture}(75.00,16.00)(-38.00,-6.00)
\put(-35.50,0.00){\makebox(0,0)[cc]{$O'_1$}}
\put(-26.50,0.00){\makebox(0,0)[cc]{$Z_1$}}
\put(-17.50,-5.00){\makebox(0,0)[cc]{$Y_1$}}
\put(-9.00,0.00){\makebox(0,0)[cc]{$Z_2$}}
\put(0.00,-4.00){\makebox(0,0)[cc]{$\cdots$}}
\put(8.50,0.00){\makebox(0,0)[cc]{$Z_{q-1}$}}
\put(17.50,-5.00){\makebox(0,0)[cc]{$Y_{q-1}$}}
\put(26.50,0.00){\makebox(0,0)[cc]{$Z_q$}}
\put(34.50,0.00){\makebox(0,0)[cc]{$Y_q$}}

\put(-35.50,8.00){\makebox(0,0)[cc]{$\{0\}\;$}}
\put(-26.50,8.00){\makebox(0,0)[cc]{$O_1$}}
\put(-9.00,8.00){\makebox(0,0)[cc]{$O_2$}}
\put(0.00,8.00){\makebox(0,0)[cc]{$\cdots$}}
\put(8.50,8.00){\makebox(0,0)[cc]{$O_{q-1}$}}
\put(26.50,8.00){\makebox(0,0)[cc]{$O_{q}$}}

\put(-33.50,0.00){\vector(1,0){5.00}}
\put(-24.50,-1.00){\vector(3,-2){4.50}}
\put(-15.50,-4.00){\vector(3,2){4.50}}
\put(-7.00,-1.00){\vector(3,-2){3.50}}
\put(2.00,-3.50){\vector(3,2){3.50}}
\put(10.50,-2.00){\vector(3,-2){3.50}}
\put(19.50,-4.00){\vector(3,2){4.00}}
\put(29.50,0.00){\vector(1,0){3.00}}

\put(-33.50,8.00){\vector(1,0){5.00}}
\put(-24.50,8.00){\vector(1,0){13.00}}
\put(-6.00,8.00){\vector(1,0){2.50}}
\put(2.00,8.00){\vector(1,0){2.50}}
\put(12.00,8.00){\vector(1,0){11.50}}

\put(-36.00,6.00){\vector(0,-1){4.00}}
\put(-27.00,6.00){\vector(0,-1){4.00}}
\put(-9.5,6.00){\vector(0,-1){4.00}}
\put(7.00,6.00){\vector(0,-1){4.00}}
\put(26.00,6.00){\vector(0,-1){4.00}}

\end{picture}

\end{tabular}
\\

\hline

\multirow{7}{*}{
\newcase\label{tablecase: red_GLxGL*}
}
&
\multirow{7}{*}{
\begin{tabular}{c}
\unitlength 0.70ex
\linethickness{0.4pt}
\begin{picture}(24.00,12.00)(-12.00,-9.00)
\put(0.00,0.00){\makebox(0,0)[cc]{
$\GL_q\times\GL_p$}
}
\put(0.00,-8.00){\makebox(0,0)[cc]
{$(\CC^{q}\otimes\CC^p)\oplus(\CC^{p})^*$}
}
\put(-5.50,-2.00){\line(-2,-3){2.50}}
\put(3.50,-2.00){\line(-2,-3){2.50}}
\put(5.50,-2.00){\line(2,-3){2.50}}
\end{picture}
\end{tabular}
}
&

$q \geq p \geq 2$

& 
\begin{tabular}{c}
$\{0\}^\vee=Y_p, (O'_1)^\vee=O_{p}, O_i^\vee=Y_{p-i}$,
$Z_i^\vee=Z_{p-i}$.
\\
\unitlength 0.70ex
\linethickness{0.4pt}
\begin{picture}(54.00,22.00)(-27.00,-19.00)
\put(-25.00,0.00){\makebox(0,0)[cc]{$\{0\}$}}
\put(-12.50,0.00){\makebox(0,0)[cc]{$O_{1}$}}
\put(0.00,0.00){\makebox(0,0)[cc]{$\cdots$}}
\put(12.50,0.00){\makebox(0,0)[cc]{$O_{p-1}$}}
\put(25.00,0.00){\makebox(0,0)[cc]{$O_{p}$}}

\put(-25.00,-9.00){\makebox(0,0)[cc]{$O'_1$}}
\put(-12.50,-9.00){\makebox(0,0)[cc]{$Z_{1}$}}
\put(0.00,-9.00){\makebox(0,0)[cc]{$\cdots$}}
\put(12.50,-9.00){\makebox(0,0)[cc]{$Z_{p-1}$}}

\put(-12.50,-18.00){\makebox(0,0)[cc]{$Y_{1}$}}
\put(0.00,-18.00){\makebox(0,0)[cc]{$\cdots$}}
\put(12.50,-18.00){\makebox(0,0)[cc]{$Y_{p-1}$}}
\put(25.00,-18.00){\makebox(0,0)[cc]{$Y_{p}$}}

\put(-20.50,0.00){\vector(1,0){4.00}}
\put(-8.00,0.00){\vector(1,0){4.00}}
\put(4.50,0.00){\vector(1,0){4.00}}
\put(17.00,0.00){\vector(1,0){4.00}}

\put(-20.50,-9.00){\vector(1,0){4.00}}
\put(-8.00,-9.00){\vector(1,0){4.00}}
\put(4.50,-9.00){\vector(1,0){4.00}}

\put(-8.00,-18.00){\vector(1,0){4.00}}
\put(4.50,-18.00){\vector(1,0){4.00}}
\put(17.00,-18.00){\vector(1,0){4.00}}

\put(-25.00,-2.50){\vector(0,-1){4.00}}
\put(-12.50,-2.50){\vector(0,-1){4.00}}
\put(-00.00,-2.50){\vector(0,-1){4.00}}
\put(12.50,-2.50){\vector(0,-1){4.00}}
\put(25.00,-2.50){\vector(0,-1){13.00}}

\put(-12.50,-11.50){\vector(0,-1){4.00}}
\put(-00.00,-11.50){\vector(0,-1){4.00}}
\put(12.50,-11.50){\vector(0,-1){4.00}}

\end{picture}  

\end{tabular}
\\

\cline{3-4}

& &

$p > q \geq 2$

&

\begin{tabular}{c}
$\{0\}^\vee=Y_q, (O'_1)^\vee=O_{q}, O_i^\vee=Y_{q-i}$,
$Z_i^\vee=Z_{q-i+1}$.
\\
\unitlength 0.70ex
\linethickness{0.4pt}
\begin{picture}(54.00,22.00)(-27.00,-19.00)
\put(-25.00,0.00){\makebox(0,0)[cc]{$\{0\}$}}
\put(-12.50,0.00){\makebox(0,0)[cc]{$O_{1}$}}
\put(0.00,0.00){\makebox(0,0)[cc]{$\cdots$}}
\put(12.50,0.00){\makebox(0,0)[cc]{$O_{q-1}$}}
\put(25.00,0.00){\makebox(0,0)[cc]{$O_{q}$}}

\put(-25.00,-9.00){\makebox(0,0)[cc]{$O'_1$}}
\put(-12.50,-9.00){\makebox(0,0)[cc]{$Z_{1}$}}
\put(0.00,-9.00){\makebox(0,0)[cc]{$\cdots$}}
\put(12.50,-9.00){\makebox(0,0)[cc]{$Z_{q-1}$}}
\put(25.00,-9.00){\makebox(0,0)[cc]{$\,Z_{q}$}}

\put(-12.50,-18.00){\makebox(0,0)[cc]{$Y_{1}$}}
\put(0.00,-18.00){\makebox(0,0)[cc]{$\cdots$}}
\put(12.50,-18.00){\makebox(0,0)[cc]{$Y_{q-1}$}}
\put(25.00,-18.00){\makebox(0,0)[cc]{$Y_{q}$}}

\put(-20.50,0.00){\vector(1,0){4.00}}
\put(-8.00,0.00){\vector(1,0){4.00}}
\put(4.50,0.00){\vector(1,0){4.00}}
\put(17.00,0.00){\vector(1,0){4.00}}

\put(-20.50,-9.00){\vector(1,0){4.00}}
\put(-8.00,-9.00){\vector(1,0){4.00}}
\put(4.50,-9.00){\vector(1,0){4.00}}
\put(17.00,-9.00){\vector(1,0){4.00}}

\put(-8.00,-18.00){\vector(1,0){4.00}}
\put(4.50,-18.00){\vector(1,0){4.00}}
\put(17.00,-18.00){\vector(1,0){4.00}}

\put(-25.00,-2.50){\vector(0,-1){4.00}}
\put(-12.50,-2.50){\vector(0,-1){4.00}}
\put(-00.00,-2.50){\vector(0,-1){4.00}}
\put(12.50,-2.50){\vector(0,-1){4.00}}
\put(25.00,-2.50){\vector(0,-1){4.00}}

\put(-12.50,-11.50){\vector(0,-1){4.00}}
\put(-00.00,-11.50){\vector(0,-1){4.00}}
\put(12.50,-11.50){\vector(0,-1){4.00}}
\put(25.00,-11.50){\vector(0,-1){4.00}}

\end{picture}   
\end{tabular}
\\

\hline

\newcase\label{tablecase: red_SpxCxC}
&
\begin{tabular}{c}
\unitlength 0.70ex
\linethickness{0.4pt}
\begin{picture}(22.00,12.00)(-11.50,-9.00)
\put(0.00,0.00){\makebox(0,0)[cc]{$\CC^\times\!\times\!\Sp_{2n}\!\times\CC^\times$}}
\put(0.00,-8.00){\makebox(0,0)[cc]{$\CC^{2n}\!\oplus\CC^{2n}$}}
\put(-9.00,-2.00){\line(2,-3){2.50}}
\put(-2.00,-2.00){\line(-2,-3){2.50}}
\put(0.00,-2.00){\line(2,-3){2.50}}
\put(7.00,-2.00){\line(-2,-3){2.50}}
\end{picture}
\end{tabular}
& $n\geq 2$ & 
\begin{tabular}{c}
$\{0\}^\vee = Y, O_1^\vee=O'_1, Z_\sim^\vee=Z_\sim,(Z^\circ)^\vee=Z^\circ$.\\
\unitlength 0.70ex
\linethickness{0.4pt}
\begin{picture}(48.00,13.00)(-24.00,-6.00)

\put(-22.00,0.00){\makebox(0,0)[cc]{$\{0\}$}}
\put(-11.00,5.00){\makebox(0,0)[cc]{$O_1$}}
\put(-11.00,-5.00){\makebox(0,0)[cc]{$O'_1$}}
\put(0.00,0.00){\makebox(0,0)[cc]{$Z_\sim$}}
\put(11.00,0.00){\makebox(0,0)[cc]{$Z^\circ$}}
\put(22.00,0.00){\makebox(0,0)[cc]{$Y$}}

\put(-19.00,1.00){\vector(3,2){5.00}}
\put(-19.00,-1.00){\vector(3,-2){5.00}}
\put(-8.00,4.00){\vector(3,-2){5.00}}
\put(-8.00,-4.00){\vector(3,2){5.00}}
\put(3.00,0.00){\vector(1,0){5.00}}
\put(14.00,0.00){\vector(1,0){5.00}}

\end{picture}   
\end{tabular}
\\

\hline

\newcase\label{tablecase: red_SpxGL}
&
\begin{tabular}{c}
\unitlength 0.70ex
\linethickness{0.4pt}
\begin{picture}(24.00,12.00)(-12.00,-9.00)
\put(0.00,0.00){\makebox(0,0)[cc]{$(\Sp_{2n}\!\times\CC^\times)\!\times\!\GL_2$}}
\put(0.00,-8.00){\makebox(0,0)[cc]{$(\CC^{2n}\!\otimes\!\CC^2)\,\oplus\,\CC^{2}$}}
\put(-3.50,-2.00){\line(0,-1){4.00}}
\put(9.50,-2.00){\line(0,-1){4.00}}
\end{picture}
\end{tabular}
& $n\geq 2$ & 

\begin{tabular}{c}
$\{0\}^\vee=Y_{22}$, $(O'_1)^\vee=O_{22}, O_{j0}^\vee=Y_{j0},Z_{10}^\vee=Z_{10}$.\\
\unitlength 0.70ex
\linethickness{0.4pt}
\begin{picture}(48.00,17.00)(-24.00,-6.00)

\put(-22.00,0.00){\makebox(0,0)[cc]{$O'_1$}}
\put(-11.00,0.00){\makebox(0,0)[cc]{$Z_{10}$}}
\put(0.00,-5.00){\makebox(0,0)[cc]{$Y_{10}$}}
\put(11.00,0.00){\makebox(0,0)[cc]{$Y_{20}$}}
\put(22.00,0.00){\makebox(0,0)[cc]{$Y_{22}$}}

\put(-22.00,8.00){\makebox(0,0)[cc]{$\{0\}$}}
\put(-11.00,8.00){\makebox(0,0)[cc]{$O_{10}$}}
\put(11.00,8.00){\makebox(0,0)[cc]{$O_{20}$}}
\put(22.00,8.00){\makebox(0,0)[cc]{$O_{22}$}}

\put(-19.00,0.00){\vector(1,0){5.00}}
\put(-8.00,-0.50){\vector(3,-2){5.00}}
\put(3.00,-4.00){\vector(3,2){5.00}}
\put(14.00,0.00){\vector(1,0){5.00}}

\put(-19.00,8.00){\vector(1,0){5.00}}
\put(-8.00,8.00){\vector(1,0){16.00}}
\put(14.00,8.00){\vector(1,0){5.00}}

\put(-22.00,6.00){\vector(0,-1){4.00}}
\put(-11.00,6.00){\vector(0,-1){4.00}}
\put(11.00,6.00){\vector(0,-1){4.00}}
\put(22.00,6.00){\vector(0,-1){4.00}}

\end{picture}   
\end{tabular}
\\

\hline

\newcase\label{tablecase: red_GLxSLxGL} &
\begin{tabular}{c}
\unitlength 0.70ex
\linethickness{0.4pt}
\begin{picture}(41.50,12.00)(-20.00,-9.00)
\put(0.00,0.00){\makebox(0,0)[cc]
{
$\GL_n\!\times\!\SL_2\!\times\!\GL_m$}
}
\put(0.00,-8.00){\makebox(0,0)[cc]
{
$(\CC^{n}\otimes\CC^2)\oplus(\CC^{2}\otimes\CC^{m})$}
}
\put(-10.50,-2.00){\line(-2,-3){2.50}}
\put(-2.00,-2.00){\line(-2,-3){2.50}}
\put(2.00,-2.00){\line(2,-3){2.50}}
\put(10.50,-2.00){\line(2,-3){2.50}}
\end{picture}
\end{tabular}
&
\renewcommand{\tabcolsep}{0pt}%
\begin{tabular}{c}
$m\geq 2$,  \\
$n\geq 2$ 
\end{tabular}
\renewcommand{\tabcolsep}{3pt}%
& 
\begin{tabular}{c}
$\{0\}^\vee=Y_{22}, (O'_1)^\vee=Y_{21}$,$(O'_2)^\vee=O_{2}$,\\
$O_{1}^\vee=Y_{12}$, the rest are self-dual.
\\
\unitlength 0.70ex
\linethickness{0.4pt}
\begin{picture}(34.00,24.00)(-17.00,-11.00)
\put(-15.00,10.00){\makebox(0,0)[cc]{$\{0\}$}}
\put(0.00,10.00){\makebox(0,0)[cc]{$O_{1}$}}
\put(15.00,10.00){\makebox(0,0)[cc]{$O_{2}$}}

\put(-15.00,0.00){\makebox(0,0)[cc]{$O'_{1}\;\;$}}
\put(-15.00,-10.00){\makebox(0,0)[cc]{$O'_{2}\;\;$}}

\put(-6.00,3.00){\makebox(0,0)[cc]{$Z_{\scriptscriptstyle 11}\;$}}
\put(6.00,-3.00){\makebox(0,0)[cc]{$Y_{\scriptscriptstyle 11}$\;}}

\put(15.00,0.00){\makebox(0,0)[cc]{$Y_{\scriptscriptstyle 21}$}}

\put(0.00,-10.00){\makebox(0,0)[cc]{$Y_{\scriptscriptstyle 12}$}}
\put(15.00,-10.00){\makebox(0,0)[cc]{$Y_{\scriptscriptstyle 22}$}}

\put(-10.00,10.00){\vector(1,0){6.00}}
\put(4.00,10.00){\vector(1,0){6.00}}

\put(-1.50,8.00){\vector(-1,-1){3.50}}
\put(-14.00,0.50){\vector(2,1){4.00}}
\put(-2.00,1.00){\vector(3,-2){4.00}}
\put(2.50,-4.50){\vector(-1,-1){3.50}}
\put(7.00,-2.00){\vector(3,1){4.50}}

\put(-10.00,-10.00){\vector(1,0){6.00}}
\put(4.00,-10.00){\vector(1,0){6.00}}

\put(-15.00,7.50){\vector(0,-1){5.00}}
\put(15.00,7.50){\vector(0,-1){5.00}}

\put(-15.00,-2.50){\vector(0,-1){5.00}}
\put(15.00,-2.50){\vector(0,-1){5.00}}

\end{picture}   
\end{tabular}
\\

\hline

\newcase\label{tablecase: red_SpxSLxGL}
&
\begin{tabular}{c}
\unitlength 0.70ex
\linethickness{0.4pt}
\begin{picture}(41.50,12.00)(-20.00,-9.00)
\put(0.00,0.00){\makebox(0,0)[cc]
{
$(\Sp_{2n}\!\times\CC^\times\!)\!\times\!\SL_2\quad\times\;\;\GL_m\quad$}
}
\put(0.00,-8.00){\makebox(0,0)[cc]
{
$(\CC^{2n}\otimes\CC^2)\quad\oplus\quad(\CC^{2}\otimes\CC^{m})$}
}
\put(-11.50,-2.00){\line(-2,-3){2.50}}
\put(-2.00,-2.00){\line(-5,-3){6.00}}
\put(2.00,-2.00){\line(5,-3){6.00}}
\put(16.00,-2.00){\line(2,-3){2.50}}
\end{picture}
\end{tabular}
&
\renewcommand{\tabcolsep}{0pt}%
\begin{tabular}{c}
     $m\geq 2$,  \\
     $n\geq 2$ 
\end{tabular}
\renewcommand{\tabcolsep}{3pt}%
& 
\begin{tabular}{c}
$\{0\}^\vee=Y_{22,2}, (O'_1)^\vee=Y_{22,1}$,$(O'_2)^\vee=O_{22}$,\\
$O_{i0}^\vee=Y_{i0,2}$, $i=1,2$,\\
the rest are self-dual.
\\
\unitlength 0.70ex
\linethickness{0.4pt}
\begin{picture}(49.00,24.00)(-17.00,-11.00)
\put(-15.00,10.00){\makebox(0,0)[cc]{$\{0\}$}}
\put(0.00,10.00){\makebox(0,0)[cc]{$O_{10}$}}
\put(15.00,10.00){\makebox(0,0)[cc]{$O_{20}$}}
\put(30.00,10.00){\makebox(0,0)[cc]{$O_{22}$}}

\put(-15.00,0.00){\makebox(0,0)[cc]{$O'_{1}\;\;$}}
\put(-15.00,-10.00){\makebox(0,0)[cc]{$O'_{2}\;\;$}}

\put(-6.00,3.00){\makebox(0,0)[cc]{$Z_{\scriptscriptstyle 10,1}\;$}}
\put(6.00,-3.00){\makebox(0,0)[cc]{$Y_{\scriptscriptstyle 10,1}$}}

\put(15.00,0.00){\makebox(0,0)[cc]{$Y_{\scriptscriptstyle 20,1}$}}
\put(30.00,0.00){\makebox(0,0)[cc]{$Y_{\scriptscriptstyle 22,1}$}}

\put(0.00,-10.00){\makebox(0,0)[cc]{$Y_{\scriptscriptstyle 10,2}$}}
\put(15.00,-10.00){\makebox(0,0)[cc]{$Y_{\scriptscriptstyle 20,2}$}}
\put(30.00,-10.00){\makebox(0,0)[cc]{$Y_{\scriptscriptstyle 22,2}$}}

\put(-10.00,10.00){\vector(1,0){6.00}}
\put(4.00,10.00){\vector(1,0){6.00}}
\put(19.00,10.00){\vector(1,0){6.00}}

\put(-1.50,8.00){\vector(-1,-1){3.50}}
\put(-14.00,0.50){\vector(2,1){4.00}}
\put(-2.00,1.00){\vector(3,-2){4.00}}
\put(2.50,-4.50){\vector(-1,-1){3.50}}
\put(7.00,-2.00){\vector(3,1){4.50}}

\put(19.00,0.00){\vector(1,0){6.00}}

\put(-10.00,-10.00){\vector(1,0){6.00}}
\put(4.00,-10.00){\vector(1,0){6.00}}
\put(19.00,-10.00){\vector(1,0){6.00}}

\put(-15.00,7.50){\vector(0,-1){5.00}}
\put(15.00,7.50){\vector(0,-1){5.00}}
\put(30.00,7.50){\vector(0,-1){5.00}}

\put(-15.00,-2.50){\vector(0,-1){5.00}}
\put(15.00,-2.50){\vector(0,-1){5.00}}
\put(30.00,-2.50){\vector(0,-1){5.00}}

\end{picture}   
\end{tabular}
\\

\hline

\newcase\label{tablecase: red_SpxSLxSp}
&
\begin{tabular}{c}
\unitlength 0.70ex
\linethickness{0.4pt}
\begin{picture}(41.50,12.00)(-20.00,-9.00)
\put(0.00,0.00){\makebox(0,0)[cc]
{
$(\Sp_{2n}\!\times\CC^\times\!)\!\times\!\SL_2\!\times(\Sp_{2m}\!\times\CC^\times\!)$}
}
\put(0.00,-8.00){\makebox(0,0)[cc]
{
$(\CC^{2n}\otimes\CC^2)\quad\oplus\quad(\CC^{2}\otimes\CC^{2m})$}
}
\put(-11.50,-2.00){\line(-2,-3){2.50}}
\put(-2.00,-2.00){\line(-5,-3){6.00}}
\put(2.00,-2.00){\line(5,-3){6.00}}
\put(16.00,-2.00){\line(2,-3){2.50}}
\end{picture}
\end{tabular}
&
\renewcommand{\tabcolsep}{0pt}%
\begin{tabular}{c}
     $m\geq 2$,  \\
     $n\geq 2$ 
\end{tabular}
\renewcommand{\tabcolsep}{3pt}%
& 
\begin{tabular}{c}
$\{0\}^\vee=Y_{22,22}$, $O_{22}^\vee=O'_{22}$,
\\$O_{i0}^\vee=Y_{i0,22}$, $(O'_{i0})^\vee=Y_{22,i0}$, $i=1,2$,\\
the rest are self-dual.
\\
\unitlength 0.70ex
\linethickness{0.4pt}
\begin{picture}(49.00,34.00)(-17.00,-21.00)
\put(-15.00,10.00){\makebox(0,0)[cc]{$\{0\}$}}
\put(0.00,10.00){\makebox(0,0)[cc]{$O_{10}$}}
\put(15.00,10.00){\makebox(0,0)[cc]{$O_{20}$}}
\put(30.00,10.00){\makebox(0,0)[cc]{$O_{22}$}}

\put(-15.00,0.00){\makebox(0,0)[cc]{$O'_{10}$}}
\put(-15.00,-10.00){\makebox(0,0)[cc]{$O'_{20}$}}

\put(-15.00,-20.00){\makebox(0,0)[cc]{$O'_{22}$}}

\put(-6.00,3.00){\makebox(0,0)[cc]{$Z_{\scriptscriptstyle 11,11}$}}
\put(6.00,-3.00){\makebox(0,0)[cc]{$Y_{\scriptscriptstyle 11,11}$}}

\put(15.00,0.00){\makebox(0,0)[cc]{$Y_{\scriptscriptstyle 20,10}$}}
\put(30.00,0.00){\makebox(0,0)[cc]{$Y_{\scriptscriptstyle 22,10}$}}

\put(0.00,-10.00){\makebox(0,0)[cc]{$Y_{\scriptscriptstyle 10,20}$}}
\put(15.00,-10.00){\makebox(0,0)[cc]{$Y_{\scriptscriptstyle 20,20}$}}
\put(30.00,-10.00){\makebox(0,0)[cc]{$Y_{\scriptscriptstyle 22,20}$}}

\put(0.00,-20.00){\makebox(0,0)[cc]{$Y_{\scriptscriptstyle 10,22}$}}
\put(15.00,-20.00){\makebox(0,0)[cc]{$Y_{\scriptscriptstyle 20,22}$}}
\put(30.00,-20.00){\makebox(0,0)[cc]{$Y_{\scriptscriptstyle 22,22}$}}

\put(-10.00,10.00){\vector(1,0){6.00}}
\put(4.00,10.00){\vector(1,0){6.00}}
\put(19.00,10.00){\vector(1,0){6.00}}

\put(-1.50,8.00){\vector(-1,-1){3.50}}
\put(-14.00,0.50){\vector(2,1){4.00}}
\put(-2.00,1.00){\vector(3,-2){4.00}}
\put(2.50,-4.50){\vector(-1,-1){3.50}}
\put(6.00,-2.00){\vector(3,1){4.50}}

\put(19.00,0.00){\vector(1,0){6.00}}

\put(-10.00,-10.00){\vector(1,0){6.00}}
\put(4.00,-10.00){\vector(1,0){6.00}}
\put(19.00,-10.00){\vector(1,0){6.00}}

\put(-10.00,-20.00){\vector(1,0){6.00}}
\put(4.00,-20.00){\vector(1,0){6.00}}
\put(19.00,-20.00){\vector(1,0){6.00}}

\put(-15.00,7.50){\vector(0,-1){5.00}}
\put(15.00,7.50){\vector(0,-1){5.00}}
\put(30.00,7.50){\vector(0,-1){5.00}}

\put(-15.00,-2.50){\vector(0,-1){5.00}}
\put(15.00,-2.50){\vector(0,-1){5.00}}
\put(30.00,-2.50){\vector(0,-1){5.00}}

\put(-15.00,-12.50){\vector(0,-1){5.00}}
\put(0.00,-12.50){\vector(0,-1){5.00}}
\put(15.00,-12.50){\vector(0,-1){5.00}}
\put(30.00,-12.50){\vector(0,-1){5.00}}

\end{picture}   
\end{tabular}
\\

\hline

\newcase\label{tablecase: red_spin8}
&
\begin{tabular}{c}
\unitlength 0.70ex
\linethickness{0.4pt}
\begin{picture}(41.50,12.00)(-20.00,-9.00)
\put(0.00,0.00){\makebox(0,0)[cc]
{
$\CC^\times\!\times\!\Spin_8\!\times\CC^\times$}
}
\put(0.00,-8.00){\makebox(0,0)[cc]
{
$\CC^8_+\;\oplus\;\CC^8_-$}
}
\put(-8.50,-2.00){\line(2,-3){2.50}}
\put(-1.00,-2.00){\line(-2,-3){2.50}}
\put(1.00,-2.00){\line(2,-3){2.50}}
\put(8.50,-2.00){\line(-2,-3){2.50}}
\end{picture}
\end{tabular}
&
$-$
& 
\begin{tabular}{c}
$\{0\}^\vee=Y_{22}, (O'_1)^\vee=Y_{21}$,$(O'_2)^\vee=O_{2}$,\\
$O_{1}^\vee=Y_{12}$, the rest are self-dual.
\\
\unitlength 0.70ex
\linethickness{0.4pt}
\begin{picture}(34.00,24.00)(-17.00,-11.00)
\put(-15.00,10.00){\makebox(0,0)[cc]{$\{0\}$}}
\put(0.00,10.00){\makebox(0,0)[cc]{$O_{1}$}}
\put(15.00,10.00){\makebox(0,0)[cc]{$O_{2}$}}

\put(-15.00,0.00){\makebox(0,0)[cc]{$O'_{1}\;\;$}}
\put(-15.00,-10.00){\makebox(0,0)[cc]{$O'_{2}\;\;$}}

\put(-6.00,3.00){\makebox(0,0)[cc]{$Z_{\scriptscriptstyle 11}\;$}}
\put(6.00,-3.00){\makebox(0,0)[cc]{$Y_{\scriptscriptstyle 11}$\;}}

\put(15.00,0.00){\makebox(0,0)[cc]{$Y_{\scriptscriptstyle 21}$}}

\put(0.00,-10.00){\makebox(0,0)[cc]{$Y_{\scriptscriptstyle 12}$}}
\put(15.00,-10.00){\makebox(0,0)[cc]{$Y_{\scriptscriptstyle 22}$}}

\put(-10.00,10.00){\vector(1,0){6.00}}
\put(4.00,10.00){\vector(1,0){6.00}}

\put(-1.50,8.00){\vector(-1,-1){3.50}}
\put(-14.00,0.50){\vector(2,1){4.00}}
\put(-2.00,1.00){\vector(3,-2){4.00}}
\put(2.50,-4.50){\vector(-1,-1){3.50}}
\put(7.00,-2.00){\vector(3,1){4.50}}

\put(-10.00,-10.00){\vector(1,0){6.00}}
\put(4.00,-10.00){\vector(1,0){6.00}}

\put(-15.00,7.50){\vector(0,-1){5.00}}
\put(15.00,7.50){\vector(0,-1){5.00}}

\put(-15.00,-2.50){\vector(0,-1){5.00}}
\put(15.00,-2.50){\vector(0,-1){5.00}}

\end{picture}   
\end{tabular}
\\

\hline

\end{longtable}
}

\section{Properties of dual orbits}
\label{s: properties_of_dual_orbits}

Groups in this section are assumed to be connected but \textit{not} necessarily reductive.

We begin by proving Lemmas~\ref{lemma: same_orbits} and~\ref{lemma: duality_and_product}.

\begin{proof}[Proof of Lemma~\ref{lemma: same_orbits}]
Any $G$-orbit $O$ in $V$ is the union of some $G'$-orbits. Since $G'$ a is normal subgroup, $G$ acts on these $G'$-obits by permutations. These smaller orbits are pairwise isomorphic and hence cannot be proper subvarieties in $O$ since $V/G'$ is finite. The claim follows.
\end{proof}

\begin{proof}[Proof of Lemma~\ref{lemma: duality_and_product}]
Clearly, $N^*(O_1\times O_2)$ decomposes into the direct sum of vector bundles $N^*O_1\oplus N^*O_2$. Moreover, on the level of subsets of $(V_1\times V_1^*)\times(V_2\times V^*_2)$ the equality $N^*(O_1\times O_2)=N^*O_1 \times N^*O_2$ holds. The latter subset has dense intersection with $N^*O_1^\vee\times N^*O_2^\vee=N^*(O_1^\vee\times O_2^\vee)$ since this is true for each factor.
\end{proof}

Suppose that $G : V$ is a linear action with finitely many orbits, and $O\subset V$ is one of them. The next lemma follows by observing that $O^\vee$ can be identified by examining a general fiber of $N^*O$.

\begin{lemma}
\label{lemma: reformulation}
Given two orbits $O\subset V$ and $Q\subset V^*$, suppose that there exists a point $x\in O$ such that the intersection $Q\cap N^*_x O$ is dense in $N^*_x O$. Then $Q=O^\vee$.
\end{lemma}
\begin{proof}
The intersection $Q\cap N^*_x O \subset \{x\}\times V^*$ is a subset of $N^*Q$. Under the diagonal action $G : V\oplus V^*$ it spans a subvariety of $N^*O\cap N^*Q$ which is dense in $\overline{N^*O}$.
\end{proof}

Suppose that there is a group $H$ and a homomorphism $G\to H$ such that the representation of $G$ is decomposed as $G\to H\to \GL(V)$. Clearly, $H$ also has finitely many orbits, and every $H$-orbit contains a unique open $G$-orbit. It turns out that duality for $H$-orbits is related to duality for some $G$-orbits:

\begin{lemma}
\label{lemma: duality_for_homomorphism}
Let $G\to H\to \GL(V)$ be representations of $G$ and $H$ with finitely many orbits. If mutually dual orbits $O\subset V$ and $O^\vee\subset V^*$ are open in $HO$ and $HO^\vee$ respectively, then  $(H\cdot O)^\vee=H\cdot O^\vee$. Conversely, if two $G$-orbits in $V$ and $V^*$ are open in mutually dual $H$-orbits, then they are dual to each other.
\end{lemma}

\begin{proof}
First, we observe that, given an open subset $Y$ in a smooth subvariety $X\subset V$, the total space $N^*Y$ of the conormal bundle is open in $N^*X\subset V\oplus V^*$.

Now, assume that $G$-orbits $O\subset V$ and $Q\subset V^*$ are open in the corresponding $H$-orbits $HO$ and $HQ$. It follows that the conormal bundles $N^*(HO)$ and $N^*(HQ)$ have dense intersection if and only if the same holds for $N^*O$ and $N^*Q$. This implies both claims.
\end{proof}

In the case where the representation $V=V_1\oplus V_2$ is reducible, the products of $G$-orbits in $V_i$ coincide with the orbits of the group $H=G\times G$. Then, along with Lemma~\ref{lemma: duality_and_product}, Lemma~\ref{lemma: reformulation} applied to the diagonal embedding $G\to H$ provides the following

\begin{lemma}
\label{lemma: duality_for_reducible}
Suppose that a group $G$ acts linearly on $V_1$ and $V_2$ with finitely many orbits, and $O_i\in V_i/G$. Then the open orbits in the products $O_1\times O_2$ and $O_1^\vee\times O_2^\vee$ are dual to each other.
\end{lemma}

We conclude the section by observing that Lemmas~\ref{lemma: duality_for_homomorphism} and~\ref{lemma: duality_for_reducible} hold true under the assumptions of Remark~\ref{rm: equivalent_orbits}, where dual orbits are assumed to lie in the same space $V$.

\section{Orbits}
\label{s: orbits}

In this section, we utilize some data listed in~\cite{Kn_MFS} for every indecomposable saturated spherical module (namely, weights of basic semi-invariants and simple reflections generating the little Weyl group, see~\S\,\ref{ss: orbits_spherical_theory}) to build an abstract orbit diagram not specifying the explicit vertex-orbit correspondence. In general, to figure out this correspondence one needs to specify basic semi-invariant functions. We develop a different approach by combining an abstract orbit diagram with an explicit description of orbits.

For every representation in Tables~\ref{table: irreducible_modules} and~\ref{table: reducible_modules} its orbit decomposition is obtained in three steps: (I) Determine an abstract orbit diagram (\S\,\ref{ss: orbits_spherical_theory}); (II) Give an explicit description of the orbits (\S\,\ref{ss: orbits_description}); (III) Combine steps I and II (\S\,\ref{ss: orbits_convergence}).

\subsection{Spherical modules}
\label{ss: orbits_spherical_theory}

We already mentioned that a spherical representation possesses only finitely many orbits. Geometry of orbits is governed by some combinatorial data. Relevant results of the theory of spherical representations are briefly formulated below. Inspection of all the cases in the tables based on these results justifies the following lemma:

\begin{lemma}
\label{lemma: abstract_orbits}
The diagram in the last column of \textup{Tables}~\textup{\ref{table: irreducible_modules}} and~\textup{\ref{table: reducible_modules}} is the abstract orbit diagram of the corresponding indecomposable saturated spherical representation.
\end{lemma}

Let $G$ be a connected reductive group. Fix a Borel subgroup $B\subset G$ along with a maximal torus $T\subset B$. Let $V$ be a finite-dimensional $G$-module. Recall that $V$ is \textit{spherical} if each simple $G$-module occurs in $\CC[V]$ with multiplicity at most one. Notice that for any $G$-submodule $V_0\subset V$ the corresponding restriction map $\CC[V]\to\CC[V_0]$ is a surjective $G$-homomorphism and hence takes isotypic components onto isotypic components. Thus we obtain

\begin{lemma}[e.g., {\cite[Lemma 2.1]{Lea}}]
\label{lemma: sphericity_for_submodules}
Every simple submodule of a spherical module is spherical.
\end{lemma}

An important geometrical characterization of spherical representations is provided by

\begin{theorem*}[{\cite{VK}}, {\cite[Theorem 3.1]{Kn_MFS}}]
A $G$-module V is spherical if and only if $B$ has an open orbit in $V$.
\end{theorem*}

Now assume that the representation $V$ is spherical. One associates with it the \textit{weight lattice} $\Lambda=\Lambda(V)$, a sublattice in the character lattice $\mathfrak{X}(B)$ of $B$ consisting of the characters $\lambda\in\mathfrak{X}(B)$ such that the subspace
\begin{equation*}
\CC(V)^{(B)}_{\lambda}=\{f\in\CC(V)\mid b\cdot f=\lambda(b)f, \forall b\in B\}
\end{equation*}
of $B$-semi-invariant rational functions with the weight $\lambda$ is nonzero. Let us also denote the corresponding dual $\QQ$-vector space by $\mathcal{E}=\mathcal{E}(V)=\Hom_\ZZ(\Lambda,\QQ)$. Put
\begin{equation*}
\mathfrak{t}(\QQ)=\{\xi\in\mathfrak{t}\mid d\lambda(\xi)\in\QQ,\forall\lambda\in\mathfrak{X}(T)\},
\end{equation*}
where $\mathfrak{t}$ denotes the Lie algebra of the torus $T$. After identifying the characters of $T$ with their differentials, and the groups $\mathfrak{X}(B)$ and $\mathfrak{X}(T)$ via restriction, $\mathfrak{t}(\QQ)^*$ serves as a natural ambient rational vector space containing the lattice $\Lambda$. The space $\mathcal{E}$, in turn, is equipped with the canonical projection from $\mathfrak{t}(\QQ)$. Since $B$ has an open orbit in $V$, there exists a unique up to proportionality rational function $f_\lambda\in\CC(V)^{(B)}_{\lambda}$ for each $\lambda\in\Lambda$. $B$-stable prime divisors in $V$ are defined by $B$-semi-invariant regular functions; weights of these functions are linearly independent and generate the lattice $\Lambda$. The rank of $\Lambda$ is called the \textit{rank} of $V$.

Each discrete $\QQ$-valued valuation of the field $\CC(V)$ vanishing on $\CC^\times$ determines an element $\overline{v}\in\mathcal{E}$, such that $\pairing{\overline{v}}{\lambda}=v(f_\lambda)$ for all $\lambda\in\Lambda$. It is known (see~\cite{Kn_LV}) that the restriction of the map $v\mapsto\overline{v}$ to the set of $G$-invariant discrete $\QQ$-valued valuations of $\CC(V)$ vanishing on $\CC^\times$ is injective. Its image $\mathcal{V}=\mathcal{V}(V)$ in $\mathcal{E}$ is a finitely generated convex cone of full dimension; it is called the \textit{valuation cone} of $V$. Primitive outer normal vectors in $\Lambda$ of the cone $\mathcal{V}$ are called \textit{spherical roots}. They are linearly independent and form the system of simple roots of some root system in $\Lambda$.

Also, there exists a subgroup $W_V$ of the Weyl group of the root system of $G$ called \textit{little Weyl group}. It acts on $\Lambda\otimes_\ZZ\QQ$ as a crystallographic reflection group preserving the lattice $\Lambda$, and the valuation cone $\mathcal{V}$ is its fundamental chamber in $\mathcal{E}$ containing the image of the negative Weyl chamber under the projection from $\mathfrak{t}(\QQ)$ (see \cite[Theorem 7.4]{Kn_ICM}). Notice that the group $W_V$ is generated by reflections along the spherical roots.

A prime divisor $D\subset V$ defines a discrete valuation $v_D$ of $\CC(V)$ by assigning to any nonzero function $f$ its order along $D$.

Let $\mathcal{D}^B=\mathcal{D}^B(V)$ (resp. $\mathcal{D}^G=\mathcal{D}^G(V)$) be the set of $B$-stable (resp. $G$-stable) prime divisors in $V$, and put $\mathcal{D}=\mathcal{D}(V)=\mathcal{D}^B\backslash\mathcal{D}^G$. Elements of the set $\mathcal{D}$ are called \textit{colors} of $V$. To a $G$-stable subvariety $Y$ we associate the set $\mathcal{D}_Y=\{D\in\mathcal{D}\mid D\supset Y\}$ as well as the convex cone $\mathcal{C}_Y\subset\mathcal{E}$ generated by $\overline{v_D}$ over all $D\subset\mathcal{D}^B$ containing $Y$.

A \textit{colored cone} is a pair $(\mathcal{C}, \mathcal{F})$ consisting of a subset $\mathcal{F}\subset\mathcal{D}$ and a convex cone $\mathcal{C}\subset\mathcal{E}$ subject to the following conditions: (i) $\mathcal{C}$ is generated by $\mathcal{F}$ and finitely many elements of $\mathcal{V}$; (ii) the relative interior of $\mathcal{C}$ intersects $\mathcal{V}$.

A \textit{colored face} of a colored cone $(\mathcal{C}, \mathcal{F})$ is a colored cone $(\mathcal{C}', \mathcal{F}')$, where $\mathcal{C}'$ is a face of $\mathcal{C}$ and $\mathcal{F}'\subset\mathcal{F}$ consists of the colors $D\in \mathcal{F}$ with $\overline{v_D}\in\mathcal{C}'$.

The representation $V$ defines a colored cone $(\mathcal{C}(V),\mathcal{D}(V))$, where $\mathcal{C}(V) = \mathcal{C}_{\{0\}}$ is the cone generated by all $\overline{v_D}$, $D\in\mathcal{D}^B$. Colored faces of this colored cone are in bijection with $G$-orbits in $V$ (see \cite[Lemma 3.2]{Kn_LV}). Namely, an orbit $O$ corresponds to the colored face $(\mathcal{C}_O, \mathcal{D}_O)$. Moreover, the inclusion $O_1\subset\overline{O_2}$ holds if and only if $\mathcal{C}_{O_2}\subset\mathcal{C}_{O_1}$.
\bigbreak
\noindent
\textit{R\'esum\'e:}
\label{resume}
Let $r$ be the rank of the spherical module $V$, and $f_{\lambda_1},\dots,f_{\lambda_r}$ be the equations of prime divisors $D_i\subset V$ lying in the complement to the open $B$-orbit. The weights $\lambda_1,\dots,\lambda_r$ form a basis of the lattice $\Lambda$. As each element $\overline{v_{D_i}}$ takes zero value on any $\lambda_j$ with $j\neq i$ and is equal to $1$ on $\lambda_i$, the cone $\mathcal{C}(V)$ is generated in $\mathcal{E}$ by vectors of the basis $v_i$ dual to $\lambda_i$. Finally, let $w_1,\dots,w_k$ denote the simple reflections generating $W_V$. The associated spherical roots determine faces of the cone $\mathcal{V}$. In~\cite{Kn_MFS} the weights $\lambda_i$ and the elements $w_j$ are listed for every indecomposable saturated spherical module. The Hasse diagram of the colored faces of $\mathcal{C}(V)$ with its arrows reversed represents the desired orbit diagram.

\begin{example}
\label{exmpl: Sp_GL3}
Consider $G=\Sp_{2n}\times\GL_3 : V=\CC^{2n}\otimes\CC^3$. Here $V$ is identified with the space of $(2n\times 3)$-matrices $x=(x_{kl})$. It is convenient to assume that the action $G : V$ is given by $(g, h)\cdot x = (\transpose{g})^{-1} x h^{-1}, g\in\Sp_{2n}, h\in\GL_3$.

Introduce the following notation: $\omega_k$ and $s_{\alpha}$ refer to a fundamental weight in Bourbaki numbering and a reflection of the Weyl group of $\Sp_{2n}$ respectively (with $\alpha$ a positive root given in the notation of~\cite{OV}); $\omega'_k$ and $s'_{ij}$ refer to the highest weight of $\wedge^k\CC^n$ and a reflection (corresponding to the root $\eps_i-\eps_j$, cf. \emph{loc. cit.}) in the Weyl group of $\GL_n$, respectively. The weights $\lambda_i$ and the reflections $w_j$ for this example are indicated in the respective entries of Table~\ref{table: exmpl_Sp_GL3}. The expressions of the spherical roots in terms of $\lambda_i$ (which are specified in the entry ``Spherical roots'') enable us to determine whether some vector in the relative interior of a given face of $\mathcal{C}(V)$ lies in the valuation cone. For instance, $\mathcal{V}$ does not intersect the relative interior of $\cone(v_1,v_2)$ since, as one can see, for any positive linear combination of $v_1$ and $v_2$ there exists a spherical root taking positive value on it.

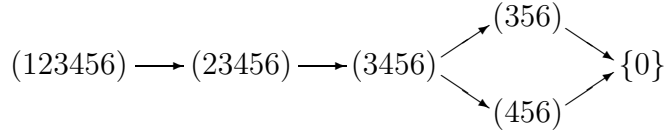
\begin{figure}
\centering
\unitlength 0.70ex
\linethickness{0.4pt}
\begin{picture}(48.00,13.00)(-24.00,-6.00)
\put(-34.00,0.00){\makebox(0,0)[cc]{$(123456)$}}
\put(-16.00,0.00){\makebox(0,0)[cc]{$(23456)$}}
\put(0.00,0.00){\makebox(0,0)[cc]{$(3456)$}}
\put(14.00,5.00){\makebox(0,0)[cc]{$(356)$}}
\put(14.00,-5.00){\makebox(0,0)[cc]{$(456)$}}
\put(26.00,0.00){\makebox(0,0)[cc]{$\{0\}$}}
\put(-27.00,0.00){\vector(1,0){5.00}}
\put(-10.00,0.00){\vector(1,0){5.00}}
\put(5.00,1.00){\vector(3,2){5.00}}
\put(5.00,-1.00){\vector(3,-2){5.00}}
\put(18.00,4.00){\vector(3,-2){5.00}}
\put(18.00,-4.00){\vector(3,2){5.00}}
\end{picture}  
\caption{An abstract orbit diagram for the case~\ref{tablecase: irr_Sp_GL3}, Table~\ref{table: reducible_modules}}
\label{fig: ex_Sp_GL3}
\end{figure}

The faces whose relative interiors intersect $\mathcal{V}$ are listed in the entry ``Colored faces'' (we use ($k_1,k_2\dots$) as a shortcut for $\cone(v_{k_1},v_{k_2},\dots)$). The reversed partial order on the set of colored faces is shown on Fig.~\ref{fig: ex_Sp_GL3}. It illustrates the claim of Lemma~\ref{lemma: abstract_orbits} for the current case. For the sake of completeness, we specify the orbits which correspond to the above listed colored faces. 

Denote bases in $\CC^3$ and $\CC^{2n}$ by $e_1,e_2,e_3$, and $a_1,\dots,a_n$, $b_n,\dots,b_1$, and assume that the symplectic form is given by the matrix
\begin{equation*}
\begin{pmatrix}
0   & J \\
-J  & 0 \\
\end{pmatrix},
\end{equation*}
where $J$ is an $(n\times n)$-matrix with $1$ on the secondary diagonal and $0$ elsewhere. This form establishes an equivariant isomorphism between $\CC^{2n}$ and $\CC^{2n*}$ and reads as $\Omega=\sum_{k=1}^n a_i\wedge b_i$ under this identification. Given $k=1,\dots,n$, put $\overline{k}=2n-k+1$. Then $x=\sum_{k,l}(x_{kl}a_k + x_{\overline{k}l}b_k)\otimes e_l$. A group $T$ (resp. $B$) in $G$ consists of the pairs of diagonal (resp. upper triangular) matrices. Let $\Delta_{l_1\dots l_r}^{k_1\dots k_r}$ be the determinant formed by the elements $x_{kl}$ with $k=k_1,\dots,k_r$ and $l=l_1,\dots,l_r$. Also put $\Omega'=(\Omega'_{ls})=\transpose{x}\Omega x$.

\begin{table}
\centering
\begin{tabular}{l l}
Basic weights: &
$\omega_1+\omega'_1$,
$\omega_2+\omega'_2$,
$\omega_3+\omega'_3$,
$\omega'_2$,
$\omega_1+\omega'_3$,
$\omega_2+\omega'_1+\omega'_3$;
\\
Basic semi-invariants: &
$\Delta_1^1$,
$\Delta_{12}^{12}$,
$\Delta_{123}^{123}$,
$\Omega'_{12}$,
$\sum_{k=2}^n \Delta_{123}^{1k\overline{k}}$,
$\Omega'_{12}\Delta_{13}^{12} - \Omega'_{13}\Delta_{12}^{12}$;
\\
Simple reflections of $W_V$: &
$s_{\eps_1-\eps_2}$,
$s_{\eps_2-\eps_3}$,
$s_{\eps_2+\eps_3}$,
$s'_{12}$,
$s'_{23}$;
\\
Spherical roots: &
\begin{tabular}[t]{l}
$\lambda_1+\lambda_5-\lambda_6$,
$-\lambda_1+\lambda_2-\lambda_3-\lambda_4+\lambda_6$,
$\lambda_3-\lambda_5$, \\
$\lambda_1-\lambda_2-\lambda_5+\lambda_6$,
$\lambda_2+\lambda_4-\lambda_6$;
\end{tabular}
\\
Colored faces: &
$(123456)$,
$(23456)$,
$(3456)$,
$(356)$,
$(456)$,
$\{0\}$;
\\
Orbits: &
$O_{00}$,
$O_{10}$,
$O_{20}$,
$O_{22}$,
$O_{30}$,
$O_{32}$.
\\
\end{tabular}

\caption{Summary for Example~\ref{exmpl: Sp_GL3}}
\label{table: exmpl_Sp_GL3}
\end{table}

Let us show that the functions from the second entry are indeed $B$-semi-invariant with desired weights. While it is obvious for the first four functions, $f_5$ and $f_6$ are still to be worked out. First, we make the following observation. It is known that, given two vector spaces $U$ and $W$, there is the following isomorphism of $\GL(U)\times\GL(W)$-modules
\begin{equation*}
\Sym^r(U\otimes W) \cong \bigoplus_\lambda \Sym^\lambda U \otimes \Sym^\lambda W.
\end{equation*}
Here, $\lambda$ runs over all the partitions of $r$, $\Sym^r$ stands for the $r$-th symmetric power, and $\Sym^\lambda$ denotes the Weyl module corresponding to $\lambda$. In particular, taking $\lambda=(1^r)$ we have $\Sym^\lambda=\wedge^r$, so $\wedge^r U^*\otimes\wedge^r W^*$ may be regarded as a submodule of $\Sym^r(U\times W)^*\subset\CC[U\otimes W]$. By this inclusion, the tensor product of two elements obtained as the wedge products of coordinate covectors of $U^*$ and $W^*$ indexed with $k_1,\dots,k_r$ and $l_1,\dots,l_r$ corresponds to the function $\Delta^{k_1,\dots,k_r}_{l_1,\dots,l_r}$.

Now, introduce $A=e^1\wedge e^2\wedge e^3$, where $e^i$ constitute the dual basis in $\CC^{3*}$, and consider the tensor $F_5 = (b_1\wedge\Omega)\otimes A$ in the space $\wedge^3\CC^{2n}\otimes\wedge^3\CC^{3*}$, which is $G$-isomorphic to $\wedge^3\CC^{2n*}\otimes\wedge^3\CC^{3*}$. Due to the fact that under the identification of the first tensor factors the vector $b_i$ corresponds to the coordinate function of the vector $a_i$ and vice versa, the value of the tensor $(b_{k_1}\wedge b_{k_2}\wedge b_{k_3})\otimes A$ at a matrix $x$ is equal to the determinant of the $(3\times 3)$-submatrix formed by the rows $k_1,k_2,k_3$ (here we assume $b_k=a_{\overline{k}}$ whenever $k>n$). Therefore, the function defined by the tensor $F_5$ is precisely $f_5$. The weight is easily calculated from the definition of $F_5$, because the form $\Omega$ is invariant, while $b_1$ and $A$ are of the weights $\omega_1$ and $\omega'_3$, correspondingly.

We now show that $f_6$ is given by the composite map $V\to\CC$ in the commutative diagram:

\begin{equation*}
\xymatrix{
& \wedge^2\CC^3 \ar[r]^{-\omega'_3 \cdot}_{\cong} & \CC^{3*} & & & \\
V \ar[ur] \ar[dr]_{-\omega_2 \cdot} \ar[rr]^{-(\omega_2+2\omega'_3)\cdot} & & \CC^{3*}\times\CC^{3*} \ar[u] \ar[d] \ar[r] & \wedge^2\CC^{3*} \ar[r]^{\omega'_3 \cdot}_{\cong} & \CC^3 \ar[r]^{-\omega'_1 \cdot} & \CC \\
& \wedge^3\CC^3 \ar[r]^{-\omega'_3 \cdot}_{\cong} & \CC^{3*} & & &
}
\end{equation*}
The maps are defined as follows. An element $x\in V$ gives rise to a linear map $\CC^{3*}\to\CC^{2n}$. In the presence of the symplectic form, denote its adjoint by $x^*:\CC^{2n}\to\CC^3$. Observe that the corresponding matrix is $-\transpose{x}\Omega$. The upper and lower diagonal arrows are given by
\begin{equation*}
x \mapsto \Omega'=x^*\Omega
\text{ and }
x \mapsto x^*b_1\wedge x^*b_2,
\end{equation*}
respectively, where $x^*$ is extended to $\wedge^2\CC^{2n}$ by putting $w_1\wedge w_2\mapsto x^* w_1\wedge x^* w_2$. The maps $\wedge^2\CC^3\to\CC^{3*}$ and $\wedge^2\CC^{3*}\to\CC^3$ are $\SL_3$-modules isomorphisms defined by the ``cross product'' formula; the vertical maps are just the projections (the top one is onto the first factor, and the bottom one is onto the second factor); the map $\CC^{3*}\times\CC^{3*}\to\wedge^2\CC^{3*}$ is given by wedge product. Finally, the function on the right is the first coordinate function with respect to the basis $e_i$. The symbols attached to arrows specify the weight by which one has to multiply the group action on the codomain in order to make the corresponding map $B$-equivariant. It is straightforward to check this in view of our convention that the actions on $\CC^{2n}$ and $\CC^3$ are defined by $g\cdot w = (\transpose{g})^{-1} w$ and $h\cdot u = (\transpose{h})^{-1} u$, $g\in\Sp_{2n}, h\in\GL_3$.

Thus, the diagram produces a semi-invariant function, and the weights along the horizontal arrows in the second row sum up to the opposite of the weight of this function. The obtained function actually equals $f_6$. This is easily seen by noting that the images of $\Omega'$ and $x^*b_1\wedge x^*b_2$ in $\CC^{3*}$ are $(\Omega'_{23},\Omega'_{31},\Omega'_{12})$ and $(\Delta_{23}^{12},\Delta_{31}^{12},\Delta_{12}^{12})$, the latter being the ``cross product'' of the first two rows of $x$ (by direct calculation). The first coordinate of the ``cross product'' of these images is exactly $\Omega'_{31}\Delta_{12}^{12}-\Omega'_{12}\Delta_{31}^{12}=f_6(x)$.

Finally, let us list the orbits. Clearly, the images of any two elements of the same orbit in $V$, regarded as linear mappings $\CC^{3*}\to\CC^{2n}$, are of the same dimension and the restrictions of the symplectic form to these images have the same rank. In other words, the following subsets are $G$-invariant:
\begin{equation*}
\begin{aligned}
O_{r0}=\{x\in V\mid \rk x = r, \Omega'(x)=0\},& \quad r=0,1,2,3; \\
O_{r2}=\{x\in V\mid \rk x = r, \Omega'(x)\neq 0\},& \quad r=2,3.
\end{aligned}
\end{equation*}
They are obviously disjoint and their number is equal to that of the vertices in the abstract orbit diagram. Thus, these subsets exhaust the list of orbits in $V$. We point out that the representation which we look at is a particular case of a family of representations of $\Sp_{2n}\times\GL_m$ which will be discussed in detail below (see~\S\,\ref{s: proof_for_irreducible}).

It is straightforward to check that the orbits in the last entry of Table~\ref{table: exmpl_Sp_GL3} agree with the cones in ``Colored faces''.

\end{example}

\subsection{Explicit orbits description}
\label{ss: orbits_description}

In the list of representations below, every item is associated to the corresponding row in Tables~\ref{table: irreducible_modules} and~\ref{table: reducible_modules}. The only exception is the item~(А\ref*{tablecase: irr_Sp_first}) corresponding to the range of rows 10--14 of Table~\ref{table: irreducible_modules}. In every case, we choose a suitable representation, which is equivalent to the one in the tables. Most of the orbital decompositions are clear or well known to experts. The following lemma deals with with the item~(A10) summarizing what will be proved in~\S\,\ref{s: proof_for_irreducible}.

\begin{lemma}
\label{lemma: prepare_Sp_GL}
The orbits of the action $\Sp_{2n}\times\GL_m : \Hom(\CC^m,\CC^{2n})$ are exactly the sets $O_{rs} = \{x: \CC^m\to \CC^{2n}\mid\rk x = r,\,\rk \Omega\vert_{\Ima x} = s\}$, where $2r-s\leq 2n,s\leq r\leq m$ and $s$ is even. Furthermore, $O_{rs}\subset \overline{O_{r's'}}$ if and only if $r\leq r'$, $s\leq s'$.
\end{lemma}

The sets listed in the entry ``Orbits'' in the remaining cases are easily shown to be disjoint and invariant under the group action. The fact that each of them is acted on transitively and that together they cover the representation space is established by matching the number of these sets with that of the vertices in the corresponding abstract orbit diagram (obtained in~\S\,\ref{ss: orbits_spherical_theory}).

\begin{remark}
\label{rm: sphericity_for_submodules}
In every row of Table~\ref{table: reducible_modules}, containing a representation of the form $V=V_1\oplus V_2$, each $V_i$ is an irreducible spherical (according to Lemma~\ref{lemma: sphericity_for_submodules}) representation of $G$. To put the representation $G : V_i$ in the form of Table~\ref{table: irreducible_modules} it suffices to: (i) remove the factors of $G$ acting on $V_i$ trivially; (ii) take the dual representation instead of the original one; (iii) add or remove a one-dimensional torus.
\end{remark}

Our notation for the orbits below follows the \hyperref[expl: tables]{explanation} on p.~\pageref{expl: tables}. In the subsection ``Reducible representations'', the product $O\times O'$ of two orbits is denoted by $X$ equipped with the corresponding subscripts. An expression of the form $O^{(...)}$(N) refers to an orbit of the case (AN) depending on the parameters $(...)$. For example, when parsing the expression $O_2^{(3,n)}(1)$, one takes the orbit $O_2$ of the case (A1) matching $\GL_q\times\GL_p$ with $q=3$, $p=n$. Except for the item (A1), where, for the sake of uniformity, we assume $p$ and $q$ to be arbitrary, all the parameters are bounded by the respective inequalities from the tables. Indices at the orbit symbols run through the admissible values in the last column of the respective row in the table. The orbits with negative subscripts are regarded as empty sets.

The elements of the nonzero orbit of minimal dimension in spinor representation are called \textit{pure spinors}.

We specify dimensions as well. They can be derived from: e.g.,~\cite{KaO} for (A1)--(A9); the proof of Theorem~\ref{thm: Sp_GL} for (A10); the items (A) for orbits $O$ and $Y$ in (B) (given $Y\subset O\times O'$, one obtains $\dim Y=\dim O + \dim O'$); the proof from \S\,\ref{s: proof_for_reducible} for orbits $Z$ in (B1)--(B4); elementary considerations for orbits $Z$ in the remaining (B) cases (the case (B10) is considered separately). 

\bigbreak
\textbf{Irreducible representations}
{
\footnotesize

\begin{longtable}[l]{c p{0.9\textwidth}}

(A\ref{tablecase: irr_GL_GL})
&
$\GL_q\times\GL_p : \Hom(\CC^p,\CC^q)$

\quad Orbits: $O_i=\{x\mid \rk x=i\}$, $i=0,\dots,\min(p,q)$.

\quad$\dim O_i=i(p+q)-i^2$. \\

(A\ref{tablecase: irr_GL_Sym})
&
$\GL_n: \Sym^2 \CC^n$

\quad Orbits: $O_i=\{x\mid \rk x =i\}$, $i=0,\dots,n$.

\quad$\dim O_i=in-i(i-1)/2$. \\

(A\ref{tablecase: irr_GL_Ext})
&
$\GL_n: \wedge^2 \CC^n$

\quad Orbits: $O_i=\{x\mid \rk x =2i\}$, $i=0,\dots,\lfloor n/2\rfloor$.

\quad$\dim O_i=i(2n-2i-1)$. \\

(A\ref{tablecase: irr_E6})
&
$\E_6\times\CC^\times: \CC^{27}$

\quad Orbits:
$O_0 = \{0\}$,
the highest weight vector orbit $O_1$,
the cubic hypersurface $O_2$,
the dense orbit $O_3$.

\quad$\dim O_i=0,17,26,27$ with $i=0,1,2,3$ respectively. \\

(A\ref{tablecase: irr_SO})
&
$\SO_n\times\CC^\times: \CC^n$

\quad Orbits:
$O_0=\{0\}$,
$O_1=$ \{non-zero isotropic vectors\},
the dense orbit $O_2$. \\

& \quad$\dim O_i=0,n-1,n$ with $i=0,1,2$ respectively. \\

(A\ref{tablecase: irr_G2})
&
$\Gtwo\times\CC^\times: \CC^7$. The group $\Gtwo$ acts on $\CC^7$ preserving a non-degenerate quadratic form. The representation is obtained by restriction from $\SO_7$. The orbits are the same as in~(A\ref*{tablecase: irr_SO}) with $n=7$. \\

(A\ref{tablecase: irr_spin7})
&
$\Spin_7\times\CC^\times: \CC^8$

\quad Orbits:
$O_0=\{0\}$,
$O_1=$ \{pure spinors\},
the dense orbit $O_2$. \\

& \quad $\dim O_i=0,7,8$ with $i=0,1,2$ respectively. \\

(A\ref{tablecase: irr_spin10})
&
$\Spin_{10}\times\CC^\times: \CC^{16}$

\quad Orbits:
$O_0=\{0\}$,
$O_1=$ \{pure spinors\},
the dense orbit $O_2$.

\quad$\dim O_i=0,11,16$ with $i=0,1,2$ respectively. \\

(A\ref{tablecase: irr_spin9})
&
$\Spin_9\times\CC^\times: \CC^{16}$

\quad This representation is the restriction of the representation~(A\ref*{tablecase: irr_spin10}). We have $O_1\text{(A\ref*{tablecase: irr_spin10})}=O_1$, $O_2\text{(A\ref*{tablecase: irr_spin10})}=O_2\cup O_3$. Let $\dotprod{\cdot}{\cdot}$ denote a bilinear form identifying the $\Spin_9$-modules $\CC^{16}$ and $(\CC^{16})^*$.

\quad Orbits:
$O_0=\{0\}$,
$O_1=$ \{pure spinors\},
$O_2=\{x\mid \dotprod{x}{x}=0\}\backslash\overline{O_1}$ (\textit{isotropic spinors}),
the dense orbit $O_3$.

\quad$\dim O_i=0,11,15,16$ with $i=0,1,2,3$ respectively. \\

(A\ref{tablecase: irr_Sp_first})
&
$\Sp_{2n}\times\GL_m : \Hom(\CC^m,\CC^{2n})$

\quad Orbits:
$O_{rs} = \{x: \CC^m\to \CC^{2n}\mid\rk x = r,\,\rk \Omega\vert_{\Ima x} = s\}$, $2r-s\leq 2n,s\leq r\leq m$, $s$ is even.

\quad $\dim O_{rs}=r(2n+m-r) - (r-s)(r-s-1)/2$.
\end{longtable}
\addtocounter{table}{-1}
}

\textbf{Reducible representations}
\label{part: reducible}

{
\footnotesize

\begin{longtable}[l]{c p{0.9\textwidth}}

(B\ref{tablecase: red_GLxC})
&
$\GL_n\times\CC^\times: \wedge^2\CC^n\oplus\CC^n$

\;
\begin{tabular}{l p{15em} p{11em} l}
Orbits:
& $O_i=O_i^{(n)}$(\ref*{tablecase: irr_GL_Ext});
& $O'_1=\CC^n\backslash\{0\}$; & \\
& $Z_i= X_{i,1}\cap\{(x,y)\mid y\perp\ker x\}$;
& $Z_{n/2}=\varnothing$, for $n$ even;
& $Y_i=X_{i,1}\backslash Z_i$. \\
\end{tabular}

\quad $\dim Z_i=i(2n-2i+1)$. \\

(B\ref{tablecase: red_GLxC*})
&
$\GL_n\times\CC^\times: \wedge^2\CC^n\oplus\CC^{n*}$

\;
\begin{tabular}{l p{15em} p{11em} l}
Orbits:
& $O_i=O_i^{(n)}$(\ref*{tablecase: irr_GL_Ext});
& $O'_1=\CC^{n*}\backslash\{0\}$; & \\
& $Z_i=X_{i,1}\cap\{(x,y)\mid y\in\ker x\}$;
& $Z_{n/2}=\varnothing$, for $n$ even;
& $Y_i=X_{i,1}\backslash Z_i$. \\
\end{tabular}

\quad $\dim Z_i=n+i(2n-2i-3)$. \\

(B\ref{tablecase: red_GLxGL})
&
$\GL_q\times\GL_p: \Hom(\CC^p,\CC^q)\oplus\CC^{p*}$

\;
\begin{tabular}{l p{15em} p{11em} l }
Orbits:
& $O_i=O_i^{(q,p)}$(\ref*{tablecase: irr_GL_GL});
& $O'_1=\CC^{p*}\backslash\{0\}$; & \\
& $Z_i= X_{i,1}\cap\{(x,y)\mid y\perp\ker x\}$;
& $Z_p=\varnothing$, if $p\leq q$;
& $Y_i=X_{i,1}\backslash Z_i$. \\
\end{tabular}

\quad $\dim Z_i=i(p+q-i+1)$. \\

(B\ref{tablecase: red_GLxGL*})
&
$\GL_q\times\GL_p: \Hom(\CC^p,\CC^q)\oplus\CC^p$

\;
\begin{tabular}{l p{15em} p{11em} l}
Orbits:
& $O_i=O_i^{(q,p)}$(\ref*{tablecase: irr_GL_GL});
& $O'_1=\CC^p\backslash\{0\}$; & \\
& $Z_i=X_{i,1}\cap\{(x,y)\mid y\in\ker x\}$;
& $Z_p=\varnothing$, if $p\leq q$;
& $Y_i=X_{i,1}\backslash Z_i$. \\
\end{tabular}

\quad $\dim Z_i=p+i(p+q-i-1)$. \\

(B\ref{tablecase: red_SpxCxC})
&
$\CC^\times\times\Sp_{2n}\times\CC^\times: \CC^{2n}\oplus\CC^{2n}$

\;
\begin{tabular}{l p{17.5em} l}
Orbits:
& $O_1, O'_1=O_1^{(2n,1)}$(\ref*{tablecase: irr_Sp_first});
& $Z_\sim=\{(x_1,x_2)\mid x_1,x_2\neq 0, x_1\sim x_2\}$; \\
& $Z^\circ=\{(x_1,x_2)\mid \symplprod{x_1}{x_2}=0\}\backslash\overline{Z_\sim}$;
& $Y=X_{1,1}\backslash(Z^\circ\cup Z_\sim)$.
\end{tabular}

Here $\sim$ denotes proportionality and $\symplprod{\cdot}{\cdot}$ denotes a fixed symplectic product on $\CC^{2n}$.

\quad $\dim Z_\sim=2n+1$; $\dim Z=4n-1$. \\

(B\ref{tablecase: red_SpxGL})
&
$(\Sp_{2n}\times\CC^\times)\times\GL_2: \Hom(\CC^2,\CC^{2n})\oplus\CC^{2}$

\;
\begin{tabular}{l p{17.5em} p{8em} l}
Orbits:
& $O_{rs}=O_{rs}^{(2n,2)}$(\ref*{tablecase: irr_Sp_first});
& $O'_1=\CC^2\backslash\{0\}$; & \\
& $Z_{10}=X_{10,1}\cap\{(x,y)\mid y\in \ker x\}$;
& $Y_{2j}=X_{2j,1}$;
& $Y_{10}=X_{1,0}\backslash Z_{10}$
\end{tabular}

\quad $\dim Z_{10}=2n+2$. \\

(B\ref{tablecase: red_GLxSLxGL})
&
$\GL_n\times\SL_2\times\GL_m:
\Hom(\CC^2,\CC^n)\oplus\Hom(\CC^m,\CC^2)$

\;
\begin{tabular}{l p{17.5em} l}
Orbits:
& $O_i=O_i^{(n,2)}$(\ref*{tablecase: irr_GL_GL});
& $O'_j=O_j^{(2,m)}$(\ref*{tablecase: irr_GL_GL}); \\
& $Z_{11}=X_{1,1}\cap\{(x,y)\mid \Ima y\subset\ker x\}$;
& $Y_{ij}=X_{i,j}$,
$(i,j)\neq(1,1)$; \\
& $Y_{11}=X_{1,1}\backslash Z_{11}$.
\end{tabular}

\quad $\dim Z_{11} = m + n + 1$. \\

(B\ref{tablecase: red_SpxSLxGL})
&
$(\Sp_{2n}\times\CC^\times)\times\SL_2\times\GL_m:
\Hom(\CC^2,\CC^{2n})\oplus\Hom(\CC^m,\CC^2)$

\;
\begin{tabular}{l p{17.5em} l}
Orbits:
& $O_{rs}=O_{rs}^{(2n,2)}$(\ref*{tablecase: irr_Sp_first});
& $O'_i=O_i^{(2,m)}$(\ref*{tablecase: irr_GL_GL}); \\
& $Z_{10,1}=X_{10,1}\cap\{(x,y)\mid \Ima y\subset\ker x\}$;
& $Y_{rs,i}=X_{rs,i}$,
$(r,s,i)\neq(1,0,1)$; \\
& $Y_{10,1}=X_{10,1}\backslash Z_{10,1}$.
\end{tabular}

\quad $\dim Z_{10,1} = 2n + m + 1$. \\

(B\ref{tablecase: red_SpxSLxSp})
&
$(\Sp_{2n}\times\CC^\times)\times\GL_2\times(\Sp_{2m}\times\CC^\times):
\Hom(\CC^2,\CC^{2n})\oplus\Hom(\CC^{2m},\CC^2)$

\;
\begin{tabular}{l p{17.5em} l}
Orbits:
& $O_{rs}=O_{rs}^{(2n,2)}$(\ref*{tablecase: irr_Sp_first});
& $O'_{r's'}=O_{r's'}^{(2,2m)}$(\ref*{tablecase: irr_Sp_first}); \\
& $Z_{10,10}=X_{10,10}\cap\{(x,y)\mid \Ima y\subset\ker x\}$;
& $Y_{rs,r's'}=X_{rs,r's'}$,
$(r,s),(r',s')\neq(1,0)$; \\
& $Y_{10,10}=X_{10,10}\backslash Z_{10,10}$.
\end{tabular}

\quad $\dim Z_{10,10}= 2(n+m) + 1$. \\

(B\ref{tablecase: red_spin8})
&
$\CC^\times\times\Spin_8\times\CC^\times: \CC^8_+\oplus\CC^8_-$

\;
\begin{tabular}{l l}
Orbits:
& $O_i,O'_i=O_i$(\ref*{tablecase: irr_spin7}); \\
& $Z_{11}$ = \{the orbit of a pair $(v_+,v_-)$ of highest-weight vectors in $\CC^8_{\pm}$\}; \\
& $Y_{ij}=X_{i,j}$,
$(i,j)\neq(1,1)$;
$Y_{11}=X_{1,1}\backslash Z_{11}$.
\end{tabular}

\quad $\dim Z_{11} = 11$.

\quad This representation is equivalent to $\CC^8\oplus\CC^8_+$. Here, the orbit $Z_{11}$ of a pair of highest weight vectors $(v,v_+)$ is a fiber bundle over the 7-dimensional orbit $O'_1$. Denote by $e_i$ a vector of weight $\eps_i$ (in the notation of~\cite{OV}) in $\CC^8$, $\pm i=1,2,3,4$. Then the isotropy group of the spinor $v_+$ acts on the space $\CC^4$ spanned by $e_1,\dots,e_4$ via the vector representation of its subgroup $\GL_4$. So the orbit $\CC^4\backslash\{0\}$ of $v=e_1$ under the highest weight spinor isotropy group is 4-dimensional. It coincides with the fiber over $v_+$, hence $\dim Z_{11}=11$. \\

\end{longtable}
\addtocounter{table}{-1}
}

\subsection{}
\label{ss: orbits_convergence}

In this subsection, we turn to a problem of assigning specific orbits to vertices of an abstract orbit diagram. Let $G : V$ be one of the representations in the tables. By $\S\,\ref{ss: orbits_spherical_theory}$, its abstract orbit diagram $\Gamma$ can be found in the tables as well. The set of orbits $V/G$ is described in $\S\,\ref{ss: orbits_description}$, its cardinality matches the number of vertices in $\Gamma$. Clearly, the maximal vertex $v_{max}$ corresponds to the open orbit, whereas the minimal vertex $v_0$ corresponds to the zero orbit. We describe what is to be done to set up the correspondence between the remaining vertices and orbits. Here we call a graph $\Gamma$ \textit{flexible} if it admits a non-trivial automorphism, i.e. a monotonic bijection of its vertices. There are four possibilities:

1) $V$ is \textit{irreducible} and

1a) $\Gamma$ is \textit{not flexible} (all rows of Table~\ref{table: irreducible_modules}, except for the 12th row). Case-by-case checking shows that the vertices of $\Gamma$ form a linearly ordered set. The explicit description in~\S\,\ref{ss: orbits_description} together with Theorem~\ref{thm: Pan} (for abelian actions, rows 1--4; see~\S\,\ref{s: proof_for_irreducible}) and Lemma~\ref{lemma: prepare_Sp_GL} (for the groups $\Sp\times\GL$) show that the order on $V/G$ is also linear. Therefore, the orbits are in unambiguous bijection with the vertices (in accordance with the growth of dimensions).

1b) $\Gamma$ is \textit{flexible}. The only relevant representation is $\Sp_{2n}\times\GL_3 : \CC^{2n}\otimes\CC^3$. In this case, the complete orbit diagram may be obtained by Theorem~\ref{thm: Sp_GL} without use of the above theory. Still, we give an overview consistent with the theory of spherical representations in Example~\ref{exmpl: Sp_GL3}.

2) $V=V_1\oplus V_2$ is \textit{reducible}. Among others, there are orbits in $V$ that lie in one of the two submodules $V_i$. In the sequel, these orbits belonging to $V_i/G$ are called \textit{pure}. The closure of a pure orbit contains only pure orbits of the same submodule. The closure of any other orbit contains at least two orbits from different submodules: this is because each orbit is \textit{biconical} (i.e. invariant under the component-wise multiplication by non-zero constants), which, in turn, is due to the saturatedness. Explicitly, if an orbit $Y\subset O\times O'$ is not pure, then $O,O'\subset\overline{Y}$. By Remark~\ref{rm: sphericity_for_submodules} and the discussion in the irreducible case, we may assume that the structure of $V_i/G$ is already known. By checking all direct summands in Table~\ref{table: reducible_modules} we see that in each case $V_i/G$ is a linearly ordered with the specified number of elements.

Now we locate the minimal vertex $v_0$ of $\Gamma$. It directly precedes two vertices $a_1$ and $b_1$, which correspond to the minimal nonzero pure orbits $O_1$ and $O'_1$ (however, the precise correspondence  is yet to be established). Consider the vertices directly following $a_1$. If, among them, there exists one with $a_1$ being the only directly preceding vertex, then it becomes our new pure vertex $a_2$. Since the order on the set of pure orbits of a fixed submodule is linear, the new one is uniquely defined. If there is no such vertex, then $a_1$ corresponds to an open orbit in one of the submodules. Proceeding in this manner, and then carrying out the same with $b_1$, we construct two linear subgraphs $\Gamma_a$ and $\Gamma_b$ starting at $v_0$ and formed by the strings of pure orbits $v_0,a_1,\dots,a_{max}$ and $v_0,b_1,\dots,b_{max}$ lying in $V_1$ and $V_2$, respectively. Suppose that:

2a) $\Gamma$ is \textit{not flexible} (all rows of~\ref{table: reducible_modules}, except for 5,7,9 and 10). In these cases, $\Gamma_a$ and $\Gamma_b$ have distinct lengths. With what we already know about the irreducible case, i.e. the number of orbits in the submodules, we may unambiguously identify these linear subgraphs with the orbit diagrams of $V_1$ and $V_2$. Assume, without loss of generality, that $\Gamma_a$ corresponds to $V_1$ and $\Gamma_b$ corresponds to $V_2$.

Let us define a partial order on the pairs of orbits $(O,O')$ given by the component-wise comparison. Then we go through the pairs starting with the maximal one, say $(O_{max},O'_{max})$. Orbits in the product $O_{max}\times O'_{max}$ correspond to the vertices $v$ subject to the inequalities $v\geq a_{max}$ and $v\geq b_{max}$. For the next pair $(O,O')$, find all the vertices of $\Gamma$ greater than those corresponding to $O$ and $O'$. After removing the vertices considered at the previous step, we are left with a set of vertices corresponding to the orbits in $O\times O'$. Proceeding like this, we uniquely determine the sets of vertices corresponding to the orbits in all products $O\times O'$. In all cases, the product $O\times O'$ is either a single orbit $Y$ or a union of two orbits $Z\cup Y$, hence the sets of vertices contain either one or two elements. Clearly, the process terminates at a uniquely determined correspondence between vertices and orbits.

2b) $\Gamma$ is \textit{flexible} (rows 5,7,9 and 10). There are two differences from the previous case: $\Gamma_a$ and $\Gamma_b$ are of the same length, and the product $O_1\times O'_1$ in the case 5 splits into three orbits (rather than one or two) somehow corresponding to the three vertices in the related set. These differences cause no trouble. Due to the natural symmetry between $V_1$ and $V_2$, it does not matter whether $V_1/G$ corresponds to $\Gamma_a$ and $V_2/G$ to $\Gamma_b$ or vice versa. In the case 5, observing that the order on the set of the three vertices is again linear still enables us to relate the vertices to the orbits in a unique way.

\begin{example}
Consider the representation~\ref{tablecase: red_SpxSLxGL} of Table~\ref{table: reducible_modules}. The action of $G$ on the submodules $V_1 = \CC^{2n}\otimes\CC^2$ and $V_2=\CC^2\otimes\CC^m$ is equivalent to the representation of the groups $\Sp_{2n}\times\GL_2$ and $\GL_2\times\GL_m$ from rows~\ref{tablecase: irr_Sp_GL2} and~\ref{tablecase: irr_GL_GL}, accordingly. The abstract orbit diagram, obtained in~\ref{ss: orbits_spherical_theory}, is shown in Figure~\ref{fig: example} (left).

We assume that the irreducible cases are already examined. Thus, we know that there are four and three linearly ordered orbits in $V_1$ and $V_2$, respectively. Starting with $v_0$, we look for all the pure vertices, i.e. the ones directly preceded by only one vertex. These vertices, forming the subgraphs $\Gamma_a$ and $\Gamma_b$, are shown in Figure~\ref{fig: example} (right) labeled by $a_i$ and $b_j$. With the notations of Table~\ref{table: irreducible_modules}, we have the following: the vertices $a_1,a_2$ and $a_3$ correspond to the orbits $O_{10},O_{20}$ and $O_{22}$ in $V_1$; the vertices $b_1$ and $b_2$ correspond to the orbits $O_1$ and $O_2$ in $V_2$.

Now we go through the pairs of vertices $(a_3,b_2), (a_2,b_2), (a_1,b_2), (a_3,b_1), (a_2,b_1), (a_1,b_1)$. Evidently, a vertex $v\geq a_3, b_2$ is unique and is equal to $v_{max}$. This means that the product $O_{22}\times O'_2$ does not split into a union of $G$-orbits, and that it coincides with $Y_{22,2}$ (see~\S\,\ref{ss: orbits_description}). The vertex set related to the pair $(a_2,b_2)$ consists of two vertices. One of them $v_{max}$ has already been considered, the other corresponds to the product $O_{20}\times O'_2 = Y_{20,2}$. The same holds up to the pair $(a_1,b_1)$, for which the related set (filled dots in Figure~\ref{fig: example} (right)) consists of two vertices. They are in an unambiguous bijection with the orbits $Z_{10,1}\leq Y_{10,1}$.

\renewcommand{\tabcolsep}{18pt}
\begin{figure}
\centering
\begin{tabular}{cc}
{
\footnotesize
\unitlength 0.70ex
\linethickness{0.4pt}
\begin{picture}(49.00,24.00)(-17.00,-11.00)
\put(-15.00,10.00){\makebox(0,0)[cc]{$\bullet$}}
\put(0.00,10.00){\makebox(0,0)[cc]{$\bullet$}}
\put(15.00,10.00){\makebox(0,0)[cc]{$\bullet$}}
\put(30.00,10.00){\makebox(0,0)[cc]{$\bullet$}}

\put(-15.00,0.00){\makebox(0,0)[cc]{$\bullet$}}
\put(-15.00,-10.00){\makebox(0,0)[cc]{$\bullet$}}

\put(-6.00,3.00){\makebox(0,0)[cc]{$\bullet\;\;\;$}}
\put(6.00,-3.00){\makebox(0,0)[cc]{$\bullet\;\;\;$}}

\put(15.00,0.00){\makebox(0,0)[cc]{$\bullet$}}
\put(30.00,0.00){\makebox(0,0)[cc]{$\bullet$}}

\put(0.00,-10.00){\makebox(0,0)[cc]{$\bullet\;\;\;$}}
\put(15.00,-10.00){\makebox(0,0)[cc]{$\bullet$}}
\put(30.00,-10.00){\makebox(0,0)[cc]{$\bullet$}}

\put(-10.00,10.00){\vector(1,0){6.00}}
\put(4.00,10.00){\vector(1,0){6.00}}
\put(19.00,10.00){\vector(1,0){6.00}}

\put(-1.50,8.00){\vector(-1,-1){3.50}}
\put(-14.00,0.50){\vector(2,1){4.00}}
\put(-2.00,1.00){\vector(3,-2){4.00}}
\put(2.50,-4.50){\vector(-1,-1){3.50}}
\put(7.00,-2.00){\vector(3,1){4.50}}

\put(19.00,0.00){\vector(1,0){6.00}}

\put(-10.00,-10.00){\vector(1,0){6.00}}
\put(4.00,-10.00){\vector(1,0){6.00}}
\put(19.00,-10.00){\vector(1,0){6.00}}

\put(-15.00,7.50){\vector(0,-1){5.00}}
\put(15.00,7.50){\vector(0,-1){5.00}}
\put(30.00,7.50){\vector(0,-1){5.00}}

\put(-15.00,-2.50){\vector(0,-1){5.00}}
\put(15.00,-2.50){\vector(0,-1){5.00}}
\put(30.00,-2.50){\vector(0,-1){5.00}}
\end{picture}
}

&

{
\footnotesize
\unitlength 0.70ex
\linethickness{0.4pt}
\begin{picture}(49.00,24.00)(-17.00,-11.00)
\put(-15.00,10.00){\makebox(0,0)[cc]{$v_0$}}
\put(0.00,10.00){\makebox(0,0)[cc]{$a_1$}}
\put(15.00,10.00){\makebox(0,0)[cc]{$a_2$}}
\put(30.00,10.00){\makebox(0,0)[cc]{$a_3$}}

\put(-15.00,0.00){\makebox(0,0)[cc]{$b_1$}}
\put(-15.00,-10.00){\makebox(0,0)[cc]{$b_2$}}

\put(-6.00,3.00){\makebox(0,0)[cc]{$\bullet\;\;\;$}}
\put(6.00,-3.00){\makebox(0,0)[cc]{$\bullet\;\;\;$}}

\put(15.00,0.00){\makebox(0,0)[cc]{$\circ$}}
\put(30.00,0.00){\makebox(0,0)[cc]{$\circ$}}

\put(0.00,-10.00){\makebox(0,0)[cc]{$\circ\;\;\;$}}
\put(15.00,-10.00){\makebox(0,0)[cc]{$\circ$}}
\put(30.00,-10.00){\makebox(0,0)[cc]{$v_{max}$}}

\put(-10.00,10.00){\vector(1,0){6.00}}
\put(4.00,10.00){\vector(1,0){6.00}}
\put(19.00,10.00){\vector(1,0){6.00}}

\put(-1.50,8.00){\vector(-1,-1){3.50}}
\put(-14.00,0.50){\vector(2,1){4.00}}
\put(-2.00,1.00){\vector(3,-2){4.00}}
\put(2.50,-4.50){\vector(-1,-1){3.50}}
\put(7.00,-2.00){\vector(3,1){4.50}}

\put(19.00,0.00){\vector(1,0){6.00}}

\put(-10.00,-10.00){\vector(1,0){6.00}}
\put(4.00,-10.00){\vector(1,0){6.00}}
\put(19.00,-10.00){\vector(1,0){6.00}}

\put(-15.00,7.50){\vector(0,-1){5.00}}
\put(15.00,7.50){\vector(0,-1){5.00}}
\put(30.00,7.50){\vector(0,-1){5.00}}

\put(-15.00,-2.50){\vector(0,-1){5.00}}
\put(15.00,-2.50){\vector(0,-1){5.00}}
\put(30.00,-2.50){\vector(0,-1){5.00}}
\end{picture}
}
\end{tabular}
\caption{The algorithm applied to the representation~\ref{tablecase: red_SpxSLxGL}, Table~\ref{table: reducible_modules}}
\label{fig: example}
\end{figure}
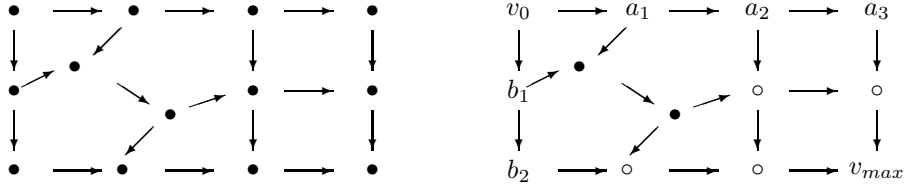
\renewcommand{\tabcolsep}{3pt}

\end{example}

\section{Proof of Theorem~\ref{thm: main_theorem}a}
\label{s: proof_for_irreducible}

Consider the irreducible representations of Table~\ref{table: irreducible_modules}.

The representations in rows 1--4 are related to abelian actions; the orbits and the Pyasetskii duality for abelian actions were studied in~\cite{Pan} (see Theorem~\ref{thm: Pan}; note that in this theorem,  as soon as $V$ and $V^*$ are identified, $O_i\subset V$ and $Q_i\subset V^*$ represent the same orbit). The case $\GL_q\times\GL_p : \CC^q\otimes\CC^p$ corresponds (up to saturation) to the representation of the subgroup $\{(g_1,g_2): \det(g_1)\det(g_2)=1\}$ provided by the short grading of $\mathfrak{sl}_{q+p}$. The representations $\GL_n : \Sym^2\CC^n,\wedge^2\CC^n$ and $\E_6\times\CC^\times : \CC^{27}$ are provided by the short gradings of the Lie algebras $\mathfrak{sp}_{2n}$, $\mathfrak{so}_{2n}$ and $\E_7$, correspondingly.

Each of the actions in rows 5--8 has three orbits, the closed and the open one being dual to each other, and the middle one being self-dual. Notice that, at the same time, the representations 5 and 8 arise as abelian actions. They correspond to the short gradings of the Lie algebras $\mathfrak{so}_{n+2}$ and $\E_6$.

The duality for the representation of the group $\Spin_9\times\CC^\times$ stems from the description of its orbits~(A\ref*{tablecase: irr_spin9}) and Lemma~\ref{lemma: duality_for_homomorphism} applied to the homomorphism $\Spin_9\times\CC^\times\to\Spin_{10}\times\CC^\times$.

\bigbreak
\noindent
\textbf{Representations 10--14.}

These cases admit a uniform description (in row 10 we assume $\CC^\times=\GL_1$), which we discuss now. The results are formulated in Theorem~\ref{thm: Sp_GL} and Corollary~\ref{cor: Sp_GL}.

Let $U$ and $W$ be vector spaces of dimensions $m$ and $2n$, and let the space $W$ be endowed with a symplectic form $\Omega=\symplprod{\cdot}{\cdot}$. It is convenient to take the representation space of the group $G=\Sp(W)\times\GL(U)$ in the form $V=\Hom(U,W)$ geometrically equivalent to $W\otimes U$. We will describe Pyasetskii duality for arbitrary $m$ and $n$. Here $V^*=\Hom(W, U)$, $\pairing{x}{y}=\tr(xy)$, $x\in V, y\in V^*$.
\begin{lemma}
\label{lemma: orbits_Sp_GL}
The sets
\begin{equation*}
O_{rs} = \{x: U\to W\mid\rk x = r,\,\rk \Omega\vert_{\Ima x} = s\}\subset V,
\end{equation*}
\begin{equation*}
Q_{kt}=\{y: W\to U\mid\dim\ker y=k,\,\dim\ker\Omega\vert_{\ker y}=t\}\subset V^*,
\end{equation*}
where $2r-s\leq 2n, s\leq r\leq m$, and $s$ is even, $t\leq k$, $t\leq 2n-k\leq m$, and $k-t$ is even, exhaust all the orbits in $V$ and $V^*$.
\end{lemma}
\begin{proof}
Put aside the restrictions for $r$ and $s$ for a moment. Clearly, the sets $O_{rs}$ are disjoint, invariant, and cover the whole $V$. At the same time, by Witt's theorem, the images of any two linear maps $x,x'\in O_{rs}$ are linked by a symplectic transformation $g$ of $W$, say, $g(\Ima x)=\Ima x'$. One can lift the corresponding map $U/{\ker x}\to U/{\ker x'}$ to a non-degenerate operator $h$ on $U$. Evidently, $x'=gxh^{-1}$, i.e. $G$ acts on $O_{rs}$ transitively.

The inequalities given above, in essence, indicate whether there exists a subspace $W_1$ in $W$ of dimension $r$ with $\rk\Omega\vert_{W_1}=s$ or not. Specifically, if such a space exists then the subspace $L=\ker\Omega\vert_{W_1}$ is isotropic (whence, lies in $W_1^\perp$), which implies that $\dim L=r-s\leq\dim W_1^\perp=2n-r$, so $2r-s\leq 2n$. Conversely, when the inequalities hold, such a subspace is easily found.

The proof for the space $V^*$ is quite similar.
\end{proof}

Denote by $\dotprod{\cdot}{\cdot}$ a non-degenerate symmetric form on $U$. To an element $x\in V$ we associate an operator $x^*\in V^*$ defined by $\symplprod{x(u)}{w}=\dotprod{u}{x^*(w)}$ for any $u\in U$, $w\in W$. The map $x\mapsto x^*$ sets up a geometric equivalence between $V$ and $V^*$. One quickly verifies that $\ker x^*=(\Ima x)^\perp=:W_0$, $\dim W_0=2n-r$, and $\ker\Omega\vert_{W_0}=W_0\cap (W_0)^\perp=\ker\Omega\vert_{\Ima x}=:L$, $\dim L = r - s$. So, under the geometric equivalence, the orbits $O_{rs}$ and $Q_{2n-r,r-s}$ (equally, $Q_{kt}$ and $O_{2n-k,2n-k-t}$) are the same.

Let us fix a non-empty orbit $O_{rs}$ and an element $x: U\to W$ in $O_{rs}$. The latter can be chosen so that the form $\dotprod{\cdot}{\cdot}$ is non-degenerate on $\ker x$. Since $\Ima x^*$ and $\ker x$ are orthogonal to each other, $U$ decomposes into their direct sum. The spaces $U$ and $W$ admit the following decompositions:
\begin{equation}\label{eq: WU_decompositions}
W = F\oplus L\oplus L'\oplus M = \Ima x\oplus L'\oplus M,~
U =\Ima x^*x\oplus x^* L' \oplus \ker x.
\end{equation}
Here, $L=\ker\Omega\vert_{\Ima x}=\Ima x\cap\ker x^*$, the spaces $F$ and $M$ are complementary to $L$ in $\Ima x$ and $\ker x^*$, respectively. The form $\Omega$ is non-degenerate on $F\oplus M$ since it is non-degenerate on each of the orthogonal direct summands. A simple calculation shows that $L$ is a \textit{Lagrangian} subspace of $(F\oplus M)^\perp$, i.e. an isotropic subspace whose dimension is half that of $(F\oplus M)^\perp$. There always exists a complementary Lagrangian subspace $L'$. Note that $\Omega$ establishes a duality between $L$ and $L'$. Further, $\Ima x^*=x^*F+x^*L'$. Moreover, this is a direct sum because some nontrivial combination of vectors of $F$ and $L'$ would lie in $\ker x^*$ otherwise. Observe that $x^*F=\Ima x^*x$.

\begin{lemma}
\label{lemma: partial_order}
The partial order on orbits in $V$ and $V^*$ is given by the rules:

1. $O_{rs}\subset\overline{O_{r's'}}$ if and only if $r\leq r'$, $s\leq s'$;

2. $Q_{k't'}\subset\overline{Q_{kt}}$ if and only if $k'\geq k$, $k'+t'\geq k+t$;
\end{lemma}
\begin{proof}
We will justify the first part of the lemma. The second one is merely a reformulation in the context of geometric equivalence. Notice that the necessity of the above inequalities follows from the closedness of the sets $\{\rk x\leq r\}$ and $\{\rk\Omega\vert_{\Ima x}\leq s\}$ (the latter being given by a closed condition $\rk\transpose{x}\Omega x\leq s$ in matrix terms).

Take $x\in O_{rs}$. We find two families $x_1(\tau)\in O_{r+1,s}$ and $x_2(\tau)\in O_{r,s+2}$, where $\tau\in\CC^\times$, which approach $x$ as $\tau$ goes to $0$. Let us retain the notation from the decompositions~(\ref{eq: WU_decompositions}).

For $x_1$, pick two arbitrary non-zero vectors $u\in\ker x$ and $w\in M$ (both subspaces are nonzero because otherwise $O_{r+1,s}$ would be empty). Choose a basis of $U$ consistent with the decomposition~(\ref{eq: WU_decompositions}) and containing $u$. We define $x_1(\tau)$ acting on the basis vectors in the same way as $x$ with one exception: $u\mapsto \tau w$. Clearly, the image of this mapping is $(r+1)$-dimensional and $\Omega$ restricts to a form of rank $s$ on it.

For $x_2$, choose a basis $l_i$ in $L$ (note that $\dim L=r-s\geq 2$ as the orbit $O_{r,s+2}$ is assumed to be non-empty), and denote by $l'_i\in L'\cong L^*$ vectors of the dual basis. The restricted map $x_2$ establishes an isomorphism $\overline{x}_2: \Ima x^* \to \Ima x = F\oplus L$, which yields a decomposition
\begin{equation*}
U = \overline{x}_2^{-1}(F)\oplus \overline{x}_2^{-1}(L)\oplus\ker x.
\end{equation*}
Define $x_2(\tau)$ acting on the first and third direct summand in the same way as $x$ and on the basis vectors $u_i=\overline{x}_2^{-1}(l_i)$ of $\overline{x}_2^{-1}(L)$ as follows:
\begin{equation*}
x_2(\tau): u_1\mapsto l_1+\tau l_2+\tau^2 l_2'+\tau l_1', u_2\mapsto l_2+\tau l_2' + \tau^2 l_1', u_i\mapsto x(u_i) \text{ when } i>2.
\end{equation*}
One readily checks that $x_2(\tau)\in O_{r,s+2}$.
\end{proof}

In the following theorem, the notation $\floortwo{N}$ stands for the greatest even integer less than or equal to $N$.

\begin{theorem}\label{thm: Sp_GL}
$O_{rs}^\vee = Q_{kt}$ if and only if $2n-k=\min(2n-r,m-r+\floortwo{r-s})$ and $\floortwo{t}=\floortwo{r-s}$ ($k-t$ is even).
\end{theorem}

\begin{proof}
Elements $y\in N^*_x O_{rs}\subset V^*$ are distinguished by the equations $\tr(y\xi x- yx\eta)=0$ for any $\xi\in\mathfrak{sp}(W)$ and $\eta\in\gl(U)$, or, equivalently, by the conditions $xy\perp_{\tr}\mathfrak{sp}(W)$ and $yx=0$. The algebra $\mathfrak{sp}(W)$ is the Lie algebra of $\Omega$-skew-symmetric operators, and its $\tr$-orthogonal complement is the space of $\Omega$-symmetric operators. In particular, $\symplprod{xy(w_1)}{w_2}=\symplprod{w_1}{xy(w_2)}$ for any $w_i\in W$. We will make use of this property in the form
\begin{equation}
\label{eq: orth_symplectic}
\dotprod{y w_1}{x^* w_2}=-\dotprod{x^* w_1}{y w_2}, \forall w_i,
\end{equation}
which is easily obtained from the definitions.

Finally, we are ready to calculate the dimensions of $\ker y$ and $\ker\Omega\vert_{\ker y}$, for general $y\in N^*_x O_{rs}$. To begin, note that the equality $yx=0$ implies that $\Ima x\subset \ker y$. Also, by taking $w_2=x u$ in~(\ref{eq: orth_symplectic}) one gets $\Ima y\subset (\Ima x^*x)^\perp=x^*L'\oplus\ker x$. Therefore,
\begin{equation*}
(y:L'\oplus M\to x^* L'\oplus\ker x) =
\begin{tabular}{c|c|c|}
\multicolumn{1}{c}{} & \multicolumn{1}{c}{$L'$} & \multicolumn{1}{c}{$M$} \\
\cline{2-3}
$x^*L'$ & $\overline{y}$ & $0$ \\
\cline{2-3}
$\ker x$ & $*$ & $*$ \\
\cline{2-3}
\multicolumn{1}{c}{} & \multicolumn{1}{c}{} & \multicolumn{1}{c}{} \\
\end{tabular}\,,
\end{equation*}
where the inclusion $y(M)\subset y(\ker x^*)\subset \ker x$ stems from~(\ref{eq: orth_symplectic}) as well, and the map $\overline{y}:L'\to x^*L'$ is obtained from $y$ by restriction to $L'$ and projection onto $x^*L'$. It follows from the same equality that the composition $\overline{y}z^{-1}$ is skew-symmetric with respect to the form $\dotprod{\cdot}{\cdot}$, where $z$ denotes the restriction $x^*\vert_{L'}$ that establishes a bijection between $L'$ and $x^*L'$. Consequently,
\begin{equation*}
\rk y\leq\min(\dim(L'\oplus M),\floortwo{\dim x^*L'}+\dim\ker x)=\min(2n-r,\floortwo{r-s}+m-r).
\end{equation*}

Now, for a sufficiently general member $y\in N^*_xO_{rs}$ the above inequality turns into an equality. By the second part of Lemma~\ref{lemma: partial_order}, for fixed $k$ one should look for an orbit $Q_{kt}\ni y$ minimizing the value of $t=\dim\ker\Omega\vert_{\ker y}$. Generally, there are two cases depending on the value of the minimum in the above inequality. If the rank $\rk y=2n-k$ equals the first argument of the minimum, then $k=r$, which means $\ker y = \Ima x$, and $t = r-s$.

If the rank $\rk y$ equals the second argument of the minimum, then the rank of $\overline{y}$ has to be the maximum possible. For a sufficiently general $y$  we may assume that the block $M\to\ker x$ is of the maximal rank. When $r$ is even, the operator $\overline{y}$ is non-degenerate, and $\ker y = \Ima x\oplus(\ker y\cap M)$, whence $\rk\Omega\vert_{\ker y}=k-t\leq s + k-r$. Thus, $t\geq r-s$. When $r$ is odd, we have $\rk\overline{y}=r-s-1$, and there exists a unique (up to proportionality) vector $l'\in L'$ with $\overline{y}l'=0$. Therefore $yl'\in\ker x$, and, as $\ker x=y(M)$, there exists a vector $w$ in $\ker y$ of the form $l' + m$, where $m\in M$. Then $\ker y = (\Ima x\oplus \CC w)\oplus(\ker y\cap M)$, and $k-t\leq (s+2) + (k-r-1)$, hence $t\geq r-s-1=\floortwo{r-s}$. Depending on the parity of $k$, the minimum possible value of $t$ is either $\floortwo{r-s}$, or $\floortwo{r-s}+1$. Anyway, $\floortwo{t}=\floortwo{r-s}$, so one unambiguously determines $t$ depending on the parity of $k-t$.

We conclude the proof by Lemma~\ref{lemma: reformulation}.
\end{proof}

By identifying geometrically equivalent orbits we claim:
\begin{corollary}
\label{cor:  Sp_GL}
The orbits $O_{r_1 s_1}$ and $O_{r_2 s_2}$ are dual to each other if and only if $\floortwo{r_1-s_1}=\floortwo{r_2-s_2}=:c$ and $r_1+r_2=\min(2n,m+c)$.
\end{corollary}

\section{Proof of Theorem~\ref{thm: main_theorem}b}
\label{s: proof_for_reducible}

Suppose that $G : V=V_1\oplus V_2$ is a representation from Table~\ref{table: reducible_modules}. Then each $V_i$ is an irreducible spherical representation of the group $G$ (see Remark~\ref{rm: sphericity_for_submodules}). By part (a) of the main theorem, we may assume the duality for orbits in these spaces to be known.

Let $O\subset V_1$, $O'\subset V_2$ be $G$-orbits. Denote by $Y(O,O')$ the open $G$-orbit in $O\times O'$. In this section such orbits are referred to as \textit{apparent}. Lemma~\ref{lemma: duality_for_reducible} shows that $Y(O,O')^\vee=Y(O^\vee,(O')^\vee)$. Clearly, all the apparent orbits are put in duality in this way. In all the cases, except for the cases 1--5, there is only one non-apparent orbit, which is inevitably self-dual.

The representation $\Sp_{2n}\times(\CC^\times)^2 :\CC^{2n}\oplus\CC^{2n}=V$ in the 5th row is a restriction of the representation $\GL_{2n}\times(\CC^\times)^2 : V$ to the subgroup. The orbits of the bigger group are $\{0\}$, $O_1$, $O'_1$, $Z_\sim$ and $Z^\circ\cup Y$. Therefore, we have $\{0\}^\vee=Z^\circ\cup Y$, $O_1^\vee=O'_1$, and the orbit $Z_\sim$ is self-dual. Lemma~\ref{lemma: duality_for_homomorphism} applied to the inclusion $\Sp_{2n}\times(\CC^\times)^2\to\GL_{2n}\times(\CC^\times)^2$ puts all the orbits in duality, except for $Z^\circ$, which has to be self-dual.

\bigbreak
\noindent
\textbf{Representations 1--4.}

We work with the representations (B1)--(B4), \S\,\ref{ss: orbits_description}; they are equivalent to the corresponding ones in Table~\ref{table: reducible_modules}. In these four cases only the non-apparent orbits $Z_i$ remain to be worked out. In the sequel, $e_k$ and $e^k$ denote vectors of dual bases of the related spaces.

\bigbreak
\noindent
\textbf{(B\ref*{tablecase: red_GLxC})}
Here $\GL_n : V=\wedge^2\CC^n\oplus\CC^n$; $V^*=\wedge^2\CC^{n*}\oplus\CC^{n*}$. We identify $\wedge^2\CC^n$ and $\wedge^2\CC^{n*}$ with the spaces of skew-symmetric $(n\times n)$-matrices, and the pairing is given by $\pairing{x'}{x}=\tr(x'x)$; elements of $\CC^n$ are represented by columns, and elements of $\CC^{n*}$ are represented by rows; $\pairing{y'}{y}=y' y=\tr(y y')$. An element $\xi\in\gl_n$ acts on an element $x\in\wedge^2\CC^n$ by the formula $\xi\cdot x=\xi x + x \transpose{\xi}$.

Fix an element $z_0=(x_0,y_0)\subset Z_i$, that is, $\rk x_0=2i, y_0\neq 0$ and $y_0\perp\ker x_0$. Elements $(x',y')\in N^*_{z_0} Z_i$ are characterized by the equation $\pairing{x'}{\xi\cdot x_0} + \pairing{y'}{\xi y_0}=0$ for all $\xi\in\gl_n$. The first term is equal to $\tr(x'(\xi x_0 + x_0\transpose{\xi}))=2\tr(x_0 x'\xi)$ since $x_0$ and $x'$ are skew-symmetric. Thus, $2\tr(x_0 x'\xi)+\tr(y_0 y' \xi)=0$ for all $\xi\in\gl_n$. Finally,
\begin{equation*}
2 x_0 x' + y_0 y' = 0.
\end{equation*}
It is easy to check that in suitable coordinates:

\begin{equation*}
x_0 =
\begin{tabular}{|c|c|}
\hline
$\Omega$ & \,$0$\, \\
\hline
\,$0$\, & $0$ \\
\hline
\end{tabular}\,,
y_0 = e_1;
x' =
\begin{tabular}{|c|c|c|}
\hline
$0$ & \;$A$\; \\
\hline
\!$-\transpose{A}$\! & $*$ \\
\hline
\end{tabular}\,,
y' =
\begin{tabular}{|c|c|}
\hline
$0$ & $-2\alpha$ \\
\hline
\end{tabular}\,,
\end{equation*}

\begin{equation*}
\text{where }
\Omega =
\begin{tabular}{|c|c|}
\hline
$0$ & \,$J$\, \\
\hline
\!$-J$\! & $0$ \\
\hline
\end{tabular}\,,
A = 
\begin{tabular}{|c|}
\hline
$0$ \\
\hline
$\alpha$ \\
\hline
\end{tabular}\,,
\alpha \in (\CC^{n-2i})^*, 
\end{equation*}
and $J$ is an $(i\times i)$-matrix with $1$ on the secondary diagonal and $0$ elsewhere. Clearly, we have $\rk x' \leq\floortwo{n-2i+1}$.

The other cases are considered similarly. Below, it is assumed that $z_0=(x_0,y_0)\in Z_i$ and that $(x',y')\in N^*_{z_0}Z_i$. We provide an equation for $x_0,y_0, x'$ and $y'$, and express these elements by matrices in suitable coordinates.

\noindent
\textbf{(B\ref*{tablecase: red_GLxC*})}
$\GL_n : V=\wedge^2\CC^n\oplus\CC^{n*}$; $V^*=\wedge^2\CC^{n*}\oplus\CC^{n}$; $2x_0 x'-y' y_0=0$;
\begin{equation*}
x_0 =
\begin{tabular}{|c|c|}
\hline
$\Omega$ & \,$0$\, \\
\hline
\,$0$\, & $0$ \\
\hline
\end{tabular}\,,
y_0= e^n;
x' =
\begin{tabular}{|c|c|c|}
\hline
$0$ & \;$A$\; \\
\hline
\!$-\transpose{A}$\! & $*$ \\
\hline
\end{tabular}\,,
y' =
\begin{tabular}{|c|}
\hline
$2\Omega\alpha$ \\
\hline
$0$ \\
\hline
\end{tabular}\,,
\text{where }
A =
\begin{tabular}{|c|c|}
\hline
$0$ & $\alpha$ \\
\hline
\end{tabular}\,,
\alpha\in\CC^{2i}.
\end{equation*}
Therefore, $\rk x' \leq \floortwo{n-2i+1}$.

\noindent
\textbf{(B\ref*{tablecase: red_GLxGL})}
$\GL_q\!\times\!\GL_p \!:\! V\!=\!\Hom(\CC^p,\CC^q)\oplus\CC^{p*}$; $V^*\!=\!\Hom(\CC^q,\CC^p)\oplus\CC^{p}$; $x_0 x'\!=\!0$, $x' x_0+y' y_0\!=\!0$;
\begin{equation*}
x_0 =
\begin{tabular}{|c|c|}
\hline
$E$ & \,$0$\, \\
\hline
\,$0$\, & $0$ \\
\hline
\end{tabular}\,,
y_0 = e^i;
x' =
\begin{tabular}{|c|c|}
\hline
\,$0$\, & \,$0$\, \\
\hline
$A$ & $*$ \\
\hline
\end{tabular}\,,
y' =
\begin{tabular}{|c|}
\hline
$0$ \\
\hline
$-\alpha$ \\
\hline
\end{tabular}\,,
\text{where }
A =
\begin{tabular}{|c|c|}
\hline
$0$ & $\alpha$ \\
\hline
\end{tabular}\,,
\alpha\in\CC^{p-i}.
\end{equation*}
Therefore, $\rk x'\leq\min(p-i,q-i+1)$.

\noindent
\textbf{(B\ref*{tablecase: red_GLxGL*})}
$\GL_q\!\times\!\GL_p \!:\! V\!=\!\Hom(\CC^p,\CC^q)\oplus\CC^p$; $V^*\!=\!\Hom(\CC^q,\CC^p)\oplus\CC^{p*}$; $x_0 x'\!=\!0, x' x_0\!=\!y_0 y';$
\begin{equation*}
x_0 =
\begin{tabular}{|c|c|}
\hline
$E$ & \,$0$\, \\
\hline
\,$0$\, & $0$ \\
\hline
\end{tabular}\,,
y_0 = e_p;
x' =
\begin{tabular}{|c|c|}
\hline
\,$0$\, & \,$0$\, \\
\hline
$A$ & $*$ \\
\hline
\end{tabular}\,,
y' =
\begin{tabular}{|c|c|}
\hline
$\alpha$ & $0$ \\
\hline
\end{tabular}\,,
\text{Where }
A =
\begin{tabular}{|c|}
\hline
$0$ \\
\hline
$\alpha$ \\
\hline
\end{tabular}\,,
\alpha\in\CC^{i*}.
\end{equation*}
Therefore, $\rk x'\leq\min(p-i,q-i+1)$.

\bigbreak
By Lemma~\ref{lemma: reformulation}, in all the above cases, the $G$-orbit $Q$ of an element $(x',y')\in N^*_{z_0}Z_i$, with $x'$ having the maximum possible rank $j$ and $y'\neq 0$, is dual to $Z_i$. Clearly, under the identification of $V$ and $V^*$, the subsets $Q$ and $Z_j$ represent the same orbit. The proof is concluded by examining the cases of even/odd $n$ and order relation between $p$ and $q$.

\end{document}